\documentclass[a4paper,11pt]{article}

\usepackage{latexsym}
\usepackage{amsmath,amssymb,color}
\usepackage{amsthm}
\usepackage[margin=3cm]{geometry}
\usepackage{graphicx}	
\usepackage[all,color]{xy}
\usepackage[abbrev]{amsrefs}

\font\sevency=wncyr7  \def\sh{\hbox{\sevency X}}

\newtheorem{teigi}{Definition}[section]
\newtheorem{teiri}[teigi]{Theorem}
\newtheorem{hodai}[teigi]{Lemma}
\newtheorem{kei}[teigi]{Corollary}
\newtheorem{mondai}[teigi]{Problem}

\title{Explicit formulas of the relation between multiple zeta functions of Arakawa-Kaneko and Euler-Zagier types}
\author{Naho Kawasaki 
\footnote{E-mail address: naho.kawasaki@hirosaki-u.ac.jp} \\
Graduate School of Science and Technology, Hirosaki University, \\
3 Bunkyo, Hirosaki, Aomori, 036-8561, Japan
}
\date{}

\begin{document}
\maketitle
\begin{abstract}
Multiple zeta functions of Arakawa-Kaneko and Euler-Zagier types are known as generalizations of the Riemann zeta function. 
In 2018, Kaneko and Tsumura proved that the multiple zeta functions of Arakawa-Kaneko type can be expressed as a ${\mathbb Q}$-linear combination of products of the ones of Euler-Zagier type and multiple zeta values. 
In this paper, we give explicit formulas to the above mentioned relation. 
Moreover, as a key of its proof, we also give certain functional equations among multi-polylogarithm functions explicitly. 
\end{abstract}

\section{Introduction}


A finite sequence ${\bf k}=(k_1,\ldots,k_r)$ of positive integers (resp. non-negative integers) is called a {\it (positive) index} (resp. {\it non-negative index}). 
An index ${\bf k}=(k_1,\ldots,k_r)$ is {\it admissible} if $k_r\geq 2$. 
For any non-negative index ${\bf e}=(e_1,\ldots,e_r)$, 
the {\it weight} and the {\it depth} of ${\bf e}$ are given by ${\rm wt}({\bf e})=e_1+\cdots+e_r$ and ${\rm dep}({\bf e})=r$, respectively. 
If an index ${\bf k}$ is of weight ${\rm wt}({\bf k})$, we also say the multiple zeta value $\zeta({\bf k})$ is of weight ${\rm wt}({\bf k})$. 
The depth $r$ can be $0$ and the unique index of depth $0$, namely the empty sequence, is denoted by $\varnothing$. 
 The index $\varnothing$ is also regarded as an admissible index.  

Multiple zeta values are given by 
$$
\zeta ( {\bf k} ) = \zeta ( k_1,\ldots,k_r ) := \sum_{0<m_1<\cdots<m_r} \frac{1}{m_1^{k_1}\cdots m_r^{k_r}} 
$$
for an admissible index ${\bf k}=(k_1,\ldots,k_r)$. 
We set $\zeta(\varnothing)=1$. 
It is known that there are many relations among multiple zeta values.  
For example, we recall duality formula for multiple zeta values. 
Let $a_1,\ldots,a_n, b_1,\ldots,b_n$ be positive integers. 
For an admissible index 
$${\bf k}=(\{1\}^{a_1-1},b_1+1,\ldots,\{1\}^{a_n-1},b_n+1),$$
 we define an admissible index ${\bf k}^\dagger$ by
$$
{\bf k}^\dagger=(\{1\}^{b_n-1},a_n+1,\ldots,\{1\}^{b_1-1},a_1+1).
$$
Then we have $\zeta({\bf k})=\zeta({\bf k}^\dagger)$. 
This formula is called duality formula for multiple zeta values. 
Positive integers $a_1,b_1,\ldots,a_n,b_n$ appears in the order on the left-hand side of the duality formula, 
and in the reverse order on the right-hand side.


Arakawa and Kaneko, in their work \cite{ak_fcn}, introduced and studied the function
\begin{equation}\label{intrep_ak}
\xi({\bf k};s ) =  \xi (k_1,\cdots,k_r;s) := \frac{1}{\Gamma(s)} \int_{0}^{\infty} \frac{
{\rm Li}(k_1,\ldots,k_r; 1-e^{-t} )}{e^t-1} t^{s-1} dt 
\end{equation}
for an index ${\bf k}=(k_1,\ldots,k_r)$ where ${\rm Li}(k_1,\ldots,k_r;z)$ is the multi-polylogarithm function defined by
\begin{equation*}
{\rm Li}({\bf k};z) =
{\rm Li}(k_1,\ldots,k_r;z) := 
\sum_{0<m_1<\cdots<m_r} \frac{z^{m_r}}{m_1^{k_1}\cdots m_r^{k_r}}
 \quad(|z|<1). 
\end{equation*}
In this paper, we call $\xi({\bf k};s)$ {\it multiple zeta functions of Arakawa-Kaneko type}. 
The integral \eqref{intrep_ak} converges for ${\rm{Re}}(s)>0$. 
When ${\bf k}=(1)$, $\xi_1(s)$ is equal to $s \zeta(s+1)$. 
The function $\xi({\bf k};s )$ continues to an entire functions of $s$.  

Arakawa and Kaneko also studied the single variable function
\begin{align}\label{sumrep_ez}
\zeta({\bf k};s)=\zeta ( k_1,\ldots,k_r;s ) := \sum_{0<m_1<\cdots<m_r} \frac{1}{m_1^{k_1}\cdots m_r^{k_r}m_{r+1}^s}
\end{align}
for an index ${\bf k}=(k_1,\ldots,k_r)$, for the purpose of establishing a connection between poly-Bernoulli numbers introduced in \cite{kaneko_pbn} and multiple zeta values. 
In this paper, we call $\xi({\bf k};s)$ {\it multiple zeta functions of Euler-Zagier type}. 
The function \eqref{sumrep_ez} is absolutely convergent for ${\rm Re}(s)>1$. 
When $r=0$, the function $\zeta ( \varnothing;s )$ is understood to be the Riemann zeta function $\zeta(s)$. 
The terminology "multiple zeta functions of Euler-Zagier type" is used for the multi-variable function
$$
\zeta_r(s_1,\ldots,s_r) := \sum_{0<m_1<\cdots<m_r} \frac{1}{m_1^{s_1}\cdots m_r^{s_r}}.  
$$
in general. 
The function $\zeta_r(s_1,\ldots,s_r)$ has been studied by many researchers (for example, \cite{hof_1992,z_1994}). 
Evaluating $s_i=k_i$ for $1\leq i \leq r$, we obtain $\zeta_{r+1}(k_1,\ldots,k_r,s_{r+1})=\zeta(k_1,\ldots,k_r;s_{r+1})$.

The multiple zeta function $\zeta ( k_1,\ldots,k_r;s )$ of Euler-Zagier type has integral representation
\begin{equation}\label{intrep_ez}
\zeta(k_1,\ldots,k_r;s ) =  \frac{1}{\Gamma(s)} \int_{0}^{\infty} \frac{
{\rm Li}(k_1,\ldots,k_r; e^{-t} )}{e^t-1} t^{s-1} dt
\end{equation}
for ${\rm{Re}}(s)>1$. 
This representation gives an analytic continuation to a meromorphic function of $s$ in the whole complex plane (\cite[Proposition 2, Theorem 3]{ak_fcn})  . 


Arakawa and Kaneko proved that the function $\xi(\{1\}^{a-1},b+1;s)$ can be written in terms of multiple zeta functions $\zeta({\bf k};s)$ of Euler-Zagier type and that the value $\xi(\{1\}^{a-1},b+1;m+1)$ has a relation like duality formula for multiple zeta values. 
 
\begin{teiri}[{\cite[Theorem 8]{ak_fcn}}]\label{akThm8}
For any positive integer $a$, non-negative integer $b$ and complex number $s$, we have
\begin{align*}
\xi(\{1\}^{a-1},b+1;s) 
&=\sum_{j=0}^{b-1}(-1)^j\zeta(\{1\}^{a-1},b+1-j)\zeta(\{1\}^{j};s) \\
&\quad +(-1)^{b} \sum_{\substack{e_1,\ldots,e_b,d \geq 0 \\ e_1+\cdots+e_{b}+d=a}} \binom{s+d-1}{d}
 \zeta(e_1+1,\ldots,e_b+1;s+d). 
\end{align*}
\end{teiri}

\begin{teiri}[{\cite[Theorem 9 (2)]{ak_fcn}}]\label{akThm9(2)}
For any positive integer $a,m$ and non-negative integer $b$, we have 
\begin{align}
&\xi(\{1\}^{a-1},b+1;m+1) - (-1)^{b} \xi(\{1\}^{m-1},b+1;a+1) \nonumber\\
\label{akDual}
&= \sum_{j=0}^{b-1} (-1)^j \zeta(\{1\}^{a-1},b+1-j) \zeta(\{1\}^{m-1},j+2).
\end{align}
\end{teiri}

On the left-hand side of \eqref{akDual}, integers $a,b,m$ in the first term appear in the order, and in the reverse order in the second term like duality formula.

Furthermore, encouraged by Theorem \ref{akThm8}, \ref{akThm9(2)}, Arakawa and Kaneko gave the following problems.
\begin{mondai}[{\cite[\S 5 (i), (ii), (iii)]{ak_fcn}}]\label{akConj}
\begin{description}
\item[(i)] For a positive index $(k_1,\ldots,k_r)$, is the function $\xi(k_1,\ldots,k_r;s)$ also expressed by multiple zeta functions as in Theorem \ref{akThm8}? 
One may deduce as a consequence further relations among multiple zeta values. 
\item[(ii)] The above problem (i) will be affirmatively answered if one can give a suitable functional equation for the multi-polylogarithm functions under the substitution $z\mapsto 1-z$. 
What can one expect about functional equations of ${\rm Li}(k_1,\ldots,k_r;z)$?
\item[(iii)] May the values at positive integers $\xi({\bf k};m+1)$ enjoy a certain kind of `duality' like the one in \ref{akThm9(2)}?    
\end{description}
\end{mondai}

Arakawa and Kaneko also obtained functional equations for the multi-polylogarithm functions of depth $1$; 
\begin{align}
{\rm Li}(k;1-z)
&=(-1)^{k-1}\sum_{i=1}^{k-1}{\rm Li}(\{1\}^{i-1},2,\{1\}^{k-1-i};z) \nonumber\\
\label{ak_dep1}
&\quad+\sum_{j=0}^{k-2}(-1)^j\zeta(k-j){\rm Li}(\{1\}^j;z)-(-1)^{k-1}\log(z){\rm Li}(\{1\}^{k-1};z)
\end{align}
(\cite[p.202, Remarks (i)]{ak_fcn}). 
However, their proof of Theorem \ref{akThm8} does not use such functional equations.

There are several results concerning Problem \ref{akConj}. 
The following lemma and theorem give partial answers to Problem \ref{akConj} (i), (ii). 
\begin{teiri}[{\cite[Lemma 3.5]{kt_fcn}}]\label{ktLem3.5}
For any index ${\bf k}$, we have 
$$
{\rm Li}({\bf k};1-z)
=\sum_{{\bf k}',d\geq0} c_{{\bf k}}({\bf k}';d) {\rm Li}(\{1\}^d;1-z)  {\rm Li}({\bf k}';z),
$$
where the sum on the right runs over indices ${\bf k}'$ and non-negative integers $d$ that satisfy ${\rm wt}({\bf k}')+d \leq {\rm wt}({\bf k})$, 
and $c_{{\bf k}}({\bf k}';d) $ is a $\mathbb{Q}$-linear combination of multiple zeta values of weight ${\rm wt}({\bf k})-{\rm wt}({\bf k}')-d$. 
We understand ${\rm Li}(\varnothing; z)=1$, ${\rm wt}(\varnothing)=0$ for the empty index $\varnothing$, 
and the constant $1$ is regarded as a multiple zeta value of weight $0$. 
\end{teiri}

\begin{teiri}[{\cite[Theorem 3.6]{kt_fcn}}]\label{ktThm3.6}
Let ${\bf k}$ be any index. 
The function $\xi({\bf k};s)$ can be written in terms of multiple zeta functions as
$$
\xi({\bf k};s)
=\sum_{{\bf k}',d\geq0} c_{{\bf k}}({\bf k}';d) \binom{s+d-1}{d}\zeta({\bf k}';s+d).
$$
Here, the sum on the right runs over indices ${\bf k}'$ and non-negative integers $d$ that satisfy ${\rm wt}({\bf k}')+d \leq {\rm wt}({\bf k})$, 
and $c_{{\bf k}}({\bf k}';d) $ is a $\mathbb{Q}$-linear combination of multiple zeta values which is given in Theorem \ref{ktLem3.5}.
\end{teiri}


Xu generalized \eqref{ak_dep1} and obtained duality-like formulas referred in Problem \ref{akConj} (iii) for certain indices.  
In particular, we can obtain Theorem \ref{akThm8} from Theorem \ref{xu(2.8)}. 
\begin{teiri}[{\cite[(2.8)]{cexu}}\label{xu(2.8)}]
For any positive integer $a$ and non-negative integer $b$, we have
\begin{align}
&{\rm Li}(\{1\}^{a-1},b+1;z) \nonumber\\
&=\sum_{j=0}^{b-1}(-1)^j \zeta(\{1\}^{a-1},b+1-j){\rm Li}(\{1\}^j;1-z) \nonumber\\
\label{cexu2.8}
&\quad+ (-1)^{b} \sum_{\substack{e_1+\cdots+e_b+d=a \\ e_1,\ldots,e_b,d\geq0}}
{\rm Li}(\{1\}^d;z){\rm Li}(e_1+1,\ldots,e_b+1;1-z) 
\end{align}
\end{teiri}

\begin{teiri}[{\cite[Theorem 3.3]{cexu}}]
For any positive integers $a,m,r$ and $k_1,\ldots,k_r\geq2$, we have
\begin{align}
&\xi(\{1\}^{a-1},k_1,\ldots,k_{r-1},k_r-1;m+1)-(-1)^{k_1+\cdots+k_r}\xi(\{1\}^{m-1},k_r,\ldots,k_2,k_1-1;a+1) \nonumber\\
&=\sum_{j=0}^{r-1}(-1)^{k_{j+2}+\cdots+k_r} \nonumber\\
&\quad\times\sum_{i=1}^{k_{j+1}-2}(-1)^{i-1}\zeta(\{1\}^{m-1},k_r,\ldots,r_{j+2},i+1)\zeta(\{1\}^{a-1},k_1,\ldots,k_j,k_{j+1}-i) \nonumber\\
&\quad +\sum_{j=0}^{r-2} (-1)^{k_j+\cdots+k_r} \left\{ \zeta(\{1\}^{a-1},k_1,\ldots,k_{j+1})\xi(\{1\}^{m-1},k_r,\ldots,k_{j+2}-1;2) \right. \nonumber\\
\label{cexuThm3.3}
&\qquad \left. -\zeta(\{1\}^{a-1},k_1,\ldots,k_j)\xi(\{1\}^{m-1},k_r,\ldots,k_{j+1}-1;2) \right\}.
\end{align}
\end{teiri}


In this paper, We give explicit formulas to Theorem \ref{ktLem3.5}, \ref{ktThm3.6} for any index. 
To state main theorems, we introduce some notation. 
For an index ${\bf k}=(k_1,k_2,\ldots, k_r)$ and a non-positive index ${\bf e}:=(e_1,\ldots,e_r)$, 
we put 
$${\bf k}_-:=
\begin{cases}
(k_1,\ldots,k_{r-1},k_r-1) & \mbox{if } k_r\geq2\\
(k_1,\ldots,k_{r-1}) & \mbox{if }k_r=1
\end{cases}
$$
 and ${\bf e}_+:=(e_1,\ldots,e_{r-1},e_r+1)$.  
We denote by ${\bf k}+{\bf e}$ the index obtained by the component-wise addition 
$$
{\bf k}+{\bf e} := (k_1+e_1,\ldots,k_r+e_r),
$$
and by $b({\bf k};{\bf e})$ the product of binomial coefficients
$$
b ( {\bf k};{\bf e} ) := \prod_{i=1}^{r} \binom{k_i+e_i-1}{e_i}. 
$$
For any index ${\bf k}=(k_1,k_2,\ldots,k_r)$, we define Hoffman's dual of ${\bf k}$ by
$${\bf k}^\vee=(\underbrace{1,\ldots,1}_{k_1}+\underbrace{1,\ldots,1}_{k_2}+\underbrace{1,\ldots,1}_{k_r}).$$
For an index ${\bf k}=(\{1\}^{a_{1}-1},b_{1}+1,\ldots,\{1\}^{a_n-1},b_n+1)$, we put
$${\bf k}^i_j=(\{1\}^{a_{i+1}-1},b_{i+1}+1,\ldots,\{1\}^{a_j-1},b_j+1)$$
and
$$\overleftarrow{\bf k}=(b_n+1,\{1\}^{a_n-1},\ldots,b_1+1,\{1\}^{a_1-1})$$
for non-negative integers $i$, $j$ with  $0\leq i<j \leq n$.  
With our convention, we have ${\bf k}^i_n=:{\bf k}^i$, ${\bf k}^0_j=:{\bf k}_j$ and ${\bf k}_0={\bf k}^n=\varnothing$. 
Hoffman's dual ${\bf k}^\vee$ can be written in terms of dual of index as ${\bf k}^\vee=\overleftarrow{(({\bf k}_+)^\dagger)_-}$, but  it is introduced to avoid complication. 

Our main theorems give more precise answers to Problem \ref{akConj} as follows.

\begin{teiri}\label{mainThm1}
Let ${\bf k}=(\{1\}^{a_1-1},b_1+1,\ldots,\{1\}^{a_n-1},b_n+1)$ with $b_n \geq 0$ be an index. 
Then, we have
\begin{align*}
&\xi({\bf k};s) \\
&=(1-\delta_{0,b_n})\zeta({\bf k}) \zeta(s) 
   -\sum_{l=1}^{n}  \sum_{j=0}^{b_l-2} (-1)^{j+{\rm wt}({\bf k}^l)} 
\zeta({\bf k}_{l-1},\{1\}^{a_l-1},b_l-j)\zeta((j+1,{\bf k}^l)^\vee;s) \\
&\quad +\sum_{l=1}^{n} (-1)^{b_l+{\rm wt}({\bf k}^l)} 
 \sum_{d= 0}^{a_l} 
 \sum_{\substack{{\rm wt}({\bf e}_1)+{\rm wt}({\bf e}_2)+d=a_l \\ {\rm dep}({\bf e}_1)=n_1,\ {\rm dep}({\bf e}_2)=n_2}}
  (-1)^{{\rm wt}({\bf e}_1)} b(({\bf k}_{l-1})^\dagger;{\bf e}_1) \\
&\quad\quad\times   
\zeta(({\bf k}_{l-1})^\dagger+{\bf e}_1) b((b_l,{\bf k}^l)^\vee;{\bf e}_2)\binom{s+d-1}{d}\zeta((b_l,{\bf k}^l)^\vee+{\bf e}_2;s+d),   
\end{align*}
where the sum on the right runs over all non-negative indices ${\bf e}_1$, ${\bf e}_2$ that satisfy ${\rm wt}({\bf e}_1)+{\rm wt}({\bf e}_2)+d=a_l$, $n_1={\rm dep}(({\bf k}_{l-1})^\dagger)$ and $n_2={\rm dep}((b_l,{\bf k}^l)^\vee)$.  
\end{teiri}


\begin{teiri}\label{mainThm2}
Let ${\bf k}=(\{1\}^{a_1-1},b_1+1,\ldots,\{1\}^{a_n-1},b_n+1)$ with $b_n \geq 0$ be an index. 
Then, we have
\begin{align*}
& {\rm Li}({\bf k};1-z) \\
%
&=(1-\delta_{0,b_n})\zeta({\bf k})
 -\sum_{l=1}^{n}  \sum_{j=0}^{b_l-2} (-1)^{j+{\rm wt}({\bf k}^l)} 
\zeta({\bf k}_{l-1},\{1\}^{a_l-1},b_l-j){\rm Li}((j+1,{\bf k}^l)^\vee;z) \\
&\quad +\sum_{l=1}^{n} (-1)^{b_l+{\rm wt}({\bf k}^l)} 
 \sum_{d= 0}^{a_l} 
 \sum_{\substack{{\rm wt}({\bf e}_1)+{\rm wt}({\bf e}_2)+d=a_l \\ {\rm dep}({\bf e}_1)=n_1,\ {\rm dep}({\bf e}_2)=n_2}}
  (-1)^{{\rm wt}({\bf e}_1)} b(({\bf k}_{l-1})^\dagger;{\bf e}_1) \\
&\quad\quad\times   
\zeta(({\bf k}_{l-1})^\dagger+{\bf e}_1) b((b_l,{\bf k}^l)^\vee;{\bf e}_2){\rm Li}(\{1\}^{d};1-z)
  {\rm Li}((b_l,{\bf k}^l)^\vee+{\bf e}_2;z). 
\end{align*}
where the sum on the right runs over all non-negative indices ${\bf e}_1$, ${\bf e}_2$ that satisfy ${\rm wt}({\bf e}_1)+{\rm wt}({\bf e}_2)+d=a_l$, $n_1={\rm dep}(({\bf k}_{l-1})^\dagger)$ and $n_2={\rm dep}((b_l,{\bf k}^l)^\vee)$.  
\end{teiri}


Theorem \ref{mainThm1} implies the following corollary, which gives a duality-like formula predicted in Problem \ref{akConj} (iii) for special values of Arakawa-Kaneko zeta functions.  
\begin{kei}\label{mainCor}
For any index ${\bf k}=(\{1\}^{a_1-1},b_1+1,\ldots,\{1\}^{a_n-1},b_n+1)$ and positive integer $m$, we have
\begin{align*}
&\xi({\bf k};m+1) - (-1)^{{\rm wt}({\bf k})-a_1} \xi (\{1\}^{m-1},\overleftarrow{(b_1,{\bf k}^1)_+};a_1+1) \\
&=(1-\delta_{0,b_n})\zeta({\bf k})  \zeta(m+1) \\
&\quad -\sum_{l=1}^{n} \sum_{j=0}^{b_l-2} (-1)^{j+{\rm wt}({\bf k}^l)}
\zeta({\bf k}_{l-1},\{1\}^{a_l-1},b_l-j)  \zeta(\{1\}^{m-1},\overleftarrow{(j+2, {\bf k}^l)_+}) \\
&\quad +\sum_{l=2}^{n} \sum_{d=0}^{a_l}  (-1)^{b_l+{\rm wt}({\bf k}^l)+d} 
\xi(({\bf k}_{l-1})_-;d+1) \xi(\{1\}^{m-1},\overleftarrow{(b_l,{\bf k}^l)_+};a_l-d+1). 
\end{align*}
\end{kei}

On the left-hand side of the equation in Theorem \ref{mainCor}, integers $a_1,b_1,\ldots,a_n,b_n,m$ in the first term appears in the order appears in the order, and in the reverse order in the $2$ term like duality formula.
We note that Theorem \ref{mainThm1} and \ref{mainThm2} generalize relations \eqref{cexuThm3.3} and \eqref{cexu2.8}, respectively. 

This paper is organized as follows. 
In \S2, we introduce a generalization of integrals on $2$-posets which are defined by Yamamoto and prepare several lemmas to show main theorems. 
In \S3, we prove main theorems. 
In \S4, we give two proofs of Corollary \ref{mainCor}.
One is to apply Theorem \ref{mainThm1} and the other is to use $2$-posets.
In \S5, we also give level $2$ versions of Theorem \ref{mainThm1}, \ref{mainThm2} and Corollary \ref{mainCor}. 

\section{Preliminaries}
\subsection{Algebraic setup}

We recall the algebraic setup for multiple zeta values which Hoffman introduced in \cite{hof_1997}. 
Let $\mathfrak{H}=\mathbb{Q} \langle e_0,e_1 \rangle$ be the non-commutative polynomial algebra of indeterminates $e_0$ and $e_1$ over $\mathbb{Q}$. 
We define the subalgebras $\mathfrak{H}^0$ and $\mathfrak{H}^1$ by
$$\mathfrak{H}^0=\mathbb{Q}+e_1 \mathfrak{H} e_0 \subset \mathfrak{H}^1=\mathbb{Q}+e_1 \mathfrak{H} \subset \mathfrak{H}.$$
We define the ${\mathbb Q}$-bilinear commutative product $\sh$ by
$$1\,\sh\, w =w \,\sh\, 1=w\quad(w\in \mathfrak{H}),$$
$$av\,\sh\, bw=a(v\,\sh\, bw)+b(av\,\sh\, w) \quad(a,b\in\{e_0,e_1\},v,w\in\mathfrak{H}).$$
The product $\sh$ called the {\it shuffle product}. 
We denote by $\mathfrak{H}_{\sh}$ the commutative $\mathbb{Q}$-algebra $\mathfrak{H}$ equipped with multiplication $\sh$. 
Then the subspace $\mathfrak{H}^1$ (resp. $\mathfrak{H}^0$) of $\mathfrak{H}$ is closed under $\sh$ and becomes a subalgebra of $\mathfrak{H}_{\sh}$ denoted by $\mathfrak{H}_{\sh}^1$ (resp. $\mathfrak{H}_{\sh}^0$). 

Each monomial $e_1e_0^{k_1-1}\cdots e_1e_0^{k_r-1}$ in $\mathfrak{H}^1$ (resp. $\mathfrak{H}^0$) corresponds to the index (resp. admissible index) $(k_1,\ldots,k_r)$. 
Using this assignment, for $0<z<1$, we obtain the $\mathbb{Q}$-linear map ${\rm Li}(\ \cdot\ ;z):\mathfrak{H}^1 \to \mathbb{R}$ which sends a monomial $e_1e_0^{k_1-1}\cdots e_1e_0^{k_r-1}$ to ${\rm Li}({\bf k};z)$ for the index ${\bf k}$. 
We also define the $\mathbb{Q}$-linear map $\zeta:\mathfrak{H}^0 \to \mathbb{R}$ similarly.  
The map ${\rm Li}(\ \cdot\ ;z):\mathfrak{H}^1_{\sh}\to\mathbb{R}$ is a $\mathbb{Q}$-algebra homomorphism; 
$${\rm Li}({\bf k}\,\sh\,{\bf l};z)={\rm Li}({\bf k};z){\rm Li}({\bf l};z)$$
for any indices ${\bf k}$ and ${\bf l}$. 
Similarly, the map $\zeta:\mathfrak{H}^0_{\sh}\to\mathbb{R}$ also gives a $\mathbb{Q}$-algebra homomorphism; 
$$\zeta({\bf k}\,\sh\,{\bf l})=\zeta({\bf k})\zeta({\bf l})$$
for any admissible indices ${\bf k}$ and ${\bf l}$. 

\subsection{$2$-posets and their associated integrals}

We introduce some integrals associated to $2$-posets which give iterated integral representations of the multi-polylogarithms and recall their properties. 
The ideas of these integrals were mainly given by Yamamoto \cite{y1}. 

\begin{teigi}[cf. {\cite[Definition 2.1]{y1}}]
\begin{itemize}
\item[(1)] A {\it $2$-poset} is a pair $X=((X,\preceq ),\delta_X )$ consisting of a finite partially ordered set $( X,\preceq )$ and a labeling map $\delta_X :X \rightarrow \{ 0,1 \}$. 
A 2-poset $X$ is called {\it semi-admissible} if $\delta_X ( x ) =1$ for all minimal elements $x\in X$. 
In addition, a semi-admissible 2-poset $X$ is called {\it admissible} if $\delta_X( x ) =0$ for all maximal elements $x\in X$. 
\item[(2)] Let $0<z<1$.  
For a semi-admissible $2$-poset $X$, we define the associated integral 
$$
I_z ( X ) = \int_{\Delta_z ( X ) } \prod_{x \in X} \omega_{\delta_X ( x )} ( t_x), 
$$
where
$$
\Delta_z ( X ) = \left\{ ( t_x )_x \in (0,z)^X \mid t_x < t_y \mbox{ if } x \prec y \right\}
$$
and
$$
\omega_0 ( t)= \frac{dt}{t}, \ \ \omega_{1} ( t) = \frac{dt}{1-t}. 
$$
\item[(3)] For an admissible $2$-poset $X$, the integral $I_z(X)$ converges as $z\rightarrow 1-0$. 
We denote this limit by $I(X)$. 
\end{itemize}
\end{teigi}

We use Hasse diagrams to indicate $2$-posets, with vertices $\circ$ and $\bullet$ corresponding to $\delta_X ( x )=0$ and $1$, respectively.  
(See {\cite[\S 2]{y1}}. )

If a semi-admissible $2$-poset $X$ is totally ordered, the integral $I_z(X)$ gives an integral representation of multi-polylogarithm functions. 
In fact, for any index $(k_1,\ldots,k_r)$, we have
\begin{equation*}
{\rm Li}(k_1,k_2,\dots,k_r;z)=
I_z\left(~
 	\begin{xy}
(	0	,	-26	)	*{\bullet};
(	4	,	-22	)	*{\circ};
(	12	,	-14	)	*{\circ};
(	16	,	-10	)	*{\bullet};
(	20	,	-6	)	*{\circ};
(	28	,	2	)	*{\circ};
(	40	,	14	)	*{\bullet};
(	44	,	18	)	*{\circ};
(	52	,	26	)	*{\circ};
(	0.45	,	-25.55	);(	3.55	,	-22.45	)**\dir{-},
(	6.00	,	-20.00	);(	10.00	,	-16.00	)**\dir{.},
(	12.55	,	-13.45	);(	15.55	,	-10.45	)**\dir{-},
(	16.55	,	-9.45	);(	19.55	,	-6.45	)**\dir{-},
(	22.00	,	-4.00	);(	26.00	,	0.00	)**\dir{.},
(	30.00	,	4.00	);(	38.00	,	12.00	)**\dir{.},
(	40.45	,	14.45	);(	43.55	,	17.55	)**\dir{-},
(	46.00	,	20.00	);(	50.00	,	24.00	)**\dir{.},
{(	0	,	-24.75	)\ar@/	^	2	mm	/@{-}	^	(	0.6	){	k_1	}(	10.75	,	-14	)},
{(	16	,	-8.75	)\ar@/	^	2	mm	/@{-}	^	(	0.6	){	k_2	}(	27	,	2	)},
{(	40	,	15.25	)\ar@/	^	2	mm	/@{-}	^	(	0.6	){	k_r	}(	50.75	,	26	)},
	\end{xy} 
~\right).
\end{equation*}
For convention, we denote the empty $2$-poset by $\varnothing$ and put $I_z(\varnothing)=1$. 
This corresponds to the fact that ${\rm Li}(\varnothing;z)=1$ for the empty index $\varnothing$. 

Similarly, the value $\xi(k_1,\ldots,k_r;m)$ at positive integer $m$ can be represented by the integral as
\begin{equation}
\label{xi_poset}
\xi(k_1,\ldots,k_r;m) = 
I \left(~
\begin{xy}
(	0	,	-22.00	)	*+{}	*{\bullet}	="k_1-l";
(	0	,	-18.00	)	*+{}	*{\circ};	
(	0	,	-10.00	)	*+{}	*{\circ}	="k_1-r";
(	0	,	2.00	)	*+{}	*{\bullet}	="k_r-l";
(	0	,	6.00	)	*+{}	*{\circ};	
(	0	,	14.00	)	*+{}	*{\circ}	="k_r-r";
(	8	,	22.00	)		*{\circ};	
(	16	,	14.00	)	*+{}	*{\bullet}	="m_top";
(	16	,	6.00	)	*+{}	*{\bullet}	="m_btm";
(	0.000	,	-21.300	);(	0.000	,	-18.600	)**\dir{-},
(	0.000	,	-16.000	);(	0.000	,	-12.000	)**\dir{.},
(	0.000	,	-8.000	);(	0.000	,	0.000	)**\dir{.},
(	0.000	,	2.700	);(	0.000	,	5.400	)**\dir{-},
(	0.000	,	8.000	);(	0.000	,	12.000	)**\dir{.},
(	0.550	,	14.550	);(	7.550	,	21.550	)**\dir{-},
(	8.450	,	21.550	);(	15.450	,	14.550	)**\dir{-},
(	16.000	,	12.000	);(	16.000	,	8.000	)**\dir{.},
{	"k_1-l"	\ar@/	^	3	mm	/@{-}	^	(	0.5	){	k_1	}	"k_1-r"	},
{	"k_r-l"	\ar@/	^	3	mm	/@{-}	^	(	0.5	){	k_r	}	"k_r-r"	},
{	"m_btm"	\ar@/	_	2.5	mm	/@{-}	_	(	0.5	){	m-1	}	"m_top"	},
\end{xy}
~\right)
\end{equation}
({\cite[Theorem 2.1]{ko_poset}}). 


We also recall an algebraic setup for $2$-posets ({\cite[Remark of \S 2]{y1}} and {\cite[p.2505]{ky})}. 
Let $\mathfrak{P}$ be the ${\mathbb Q}$-algebra generated by the isomorphism classes of $2$-posets, 
whose multiplication is given by the disjoint union $2$-posets. 
We set the subalgebra ${\mathfrak P}^1$ (resp. ${\mathfrak P}^0$) of ${\mathfrak P}$ generated by the classes of semi-admissible (resp. admissible) $2$-posets. 
Then the integral defines ${\mathbb Q}$-algebra homomorphisms $I_z:\mathfrak{P}^1 \rightarrow {\mathbb R}$ and $I:\mathfrak{P}^0 \rightarrow {\mathbb R}$. 
Moreover, there is a unique ${\mathbb Q}$-algebra homomorphism $W: {\mathfrak P}\rightarrow {\mathfrak H}_{\sh}$ which satisfies the following two condition: 
\begin{itemize}
\item[(1)] If a $2$-poset $X=\{x_1 \prec x_2 \prec \cdots \prec x_k \}$ is totally ordered, 
$$W(X)=e_{\delta_X (x_1)} e_{\delta_X (x_2)} \cdots e_{\delta_X (x_k)}.$$   
\item[(2)]
For non-comparable elements $a$, $b$ of  a $2$-poset $X$,
\begin{equation}\label{xab}
W(X)=W(X_a^b)+W(X_b^a). 
\end{equation}
Here, $X_a^b$ denotes the $2$-poset that is obtained from $X$ by adjoining the relation $a \prec b$. 
\end{itemize}
We also have $W({\mathfrak P})={\mathfrak H}_{\sh}$, $I_z={\rm Li}(W(\ \cdot \ );z): {\mathfrak P}^1\rightarrow {\mathbb R}$ and $I=\zeta \circ W: {\mathfrak P}^0\rightarrow {\mathbb R}$. 

For a $2$-poset $X=((X,\preceq),\delta_X)$, let $X^t=((X,\preceq^t), \delta_{X^t})$ be its {\it transpose} which satisfy $y \preceq^\dagger x$ if $x \preceq y$ and $\delta_{X^t}=1-\delta_X$. 
We extend the map $X \mapsto X^t$ to a ${\mathbb Q}$-linear automorphism on ${\mathfrak P}$. 
The following equalities are easily verified from the definition;
$$W(X)^\dagger=W(X^t)$$ 
for any $X \in {\mathfrak P}^0$.  

For an index ${\bf k}=(\{1\}^{a_1-1},b_1+1,\ldots,\{1\}^{a_n-1},b_n+1)$, we write
$$
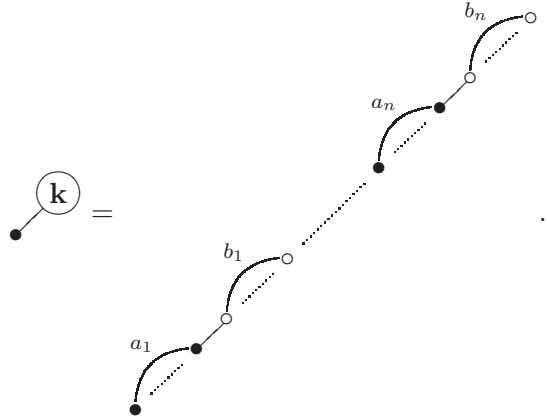
\begin{xy}
(0,-2.8)			*+	{}						*{\bullet};
(5.6,2.8)			*++	{\bf k}						*\frm{o};
(	0.5800	,	-2.220	);(	3.580	,	0.780	)**\dir{-},
\end{xy}
=
\begin{xy}
(0,-26)			*+	{}						*{\bullet}	="a1_l";
(8,-18)			*+	{}						*{\bullet}	="a1_r";
{(0,-24.8) \ar @/^2mm/ @{-}^(0.6){a_{1}} (6.8,-17.8)}, 
(12,-14)			*+	{}						*{\circ}	="b1_l";
(20,-6)			*+	{}						*{\circ}	="b1_r";
{(12,-12.8) \ar @/^2mm/ @{-}^(0.6){b_{1}} (18.8,-5.8)}, 
(32,6)			*+	{}						*{\bullet}	="as_l";
(40,14)			*+	{}						*{\bullet}	="as_r";
{(32,7.2) \ar @/^2mm/ @{-}^(0.6){a_{n}} (38.8,14.2)},
(44,18)			*+	{}						*{\circ}	="bs_l";
(52,26)			*+	{}						*{\circ}	="bs_r";
{(44,19.2) \ar @/^2mm/ @{-}^(0.6){b_{n}} (50.8,26.2)},
(2.08,-23.92);(6.08,-19.92) **\dir{.},
(8.58,-17.42);(11.58,-14.42) **\dir{-},
(14.08,-11.92);(18.08,-7.92) **\dir{.},
(22.08,-3.92);(30.08,4.08) **\dir{.},
(34.08,8.08);(38.08,12.08) **\dir{.}, 
(40.58,14.58);(43.58,17.58) **\dir{-},
(46.08,20.08);(50.08,24.08) **\dir{.},
\end{xy}.
$$
%

\subsection{Lemmas}

To prove Theorem \ref{mainThm2}, we prepare several lemmas. 

\begin{hodai}[{\cite[Lemma 1 (i)]{ak_fcn}}]\label{hodai2.3}
For any index $(k_1,\ldots,k_r)$, we have
$$
\frac{d}{dz} {\rm Li}(k_1,\ldots,k_r;z)
=
\begin{cases}
\displaystyle{\frac{1}{z}{\rm Li}(k_1,\ldots,k_{r-1},k_r-1;z)} &\quad (k_r\geq2), \\
\\
\displaystyle{\frac{1}{1-z}{\rm Li}(k_1,\ldots,k_{r-1};z)} &\quad (k_r=1). 
\end{cases}
$$

\end{hodai}

\begin{hodai}[{\cite[Theorem 2.5]{kt_fcn}}, {\cite[Theorem 2.2]{ko_poset}}]\label{kt_fcn}
For any positive index ${\bf k}$ and any positive integer $m$, we have
\begin{equation*}\label{xirel}
\xi({\bf k};m) = \sum_{\substack{{\rm wt}({\bf j})=m-1 \\ {\rm dep}({\bf j})=n}} b(({\bf k}_+)^\dagger;{\bf j}) \zeta(({\bf k}_+)^\dagger+{\bf j})
\end{equation*}
where sum is over all non-negative indices ${\bf j}$ of weight $m-1$ and depth $n={\rm dep}(({\bf k}_+)^\dagger)$. 
\end{hodai}

\begin{hodai}\label{hodai_lim}
For  any admissible index ${\bf l}$ and non-negative integer $a$, we have 
\begin{align*}
&\lim_{z\to 1 }\left\{ {\rm Li}({\bf l},\{1\}^a;z) - \sum_{d= 1}^a (-1)^{a-d} \sum_{\substack{{\rm wt}({\bf e})+d=a \\ {\rm dep}({\bf e})=n }}  b({\bf l}^\dagger;{\bf e}) \zeta({\bf l}^\dagger+{\bf e}) {\rm Li}(\{1\}^d;z) \right\} \\
&= (-1)^a \sum_{\substack{{\rm wt}({\bf e})=a \\ {\rm dep}({\bf e})=n }}  b({\bf l}^\dagger;{\bf e}) \zeta({\bf l}^\dagger+{\bf e}).
\end{align*}
\end{hodai}

\begin{proof}[Proof of Lemma \ref{hodai_lim}]
For $i,j\geq0$ and $1\leq d \leq a-1$, let
$$
v_{i,j}=
\begin{xy}
(	0	,	2.00	)	*+{}	                                                            *{\bullet};
(	0	,	  8.00	)	*++{\bf l}	                                                    *\frm{o};
(	8.000	,        0.000	)	*+{}	                                                            *{\bullet}      ="h1";	
(	8.000	,       -  8.000	)	*+{}						                 *{\bullet}	="hh";
(	16.000	,        0.000	)	*+{}	                                                            *{\bullet}      ="d1";	
(	16.000	,         8.000	)	*+{}						                 *{\bullet}	="dd";
(	0.000	,        2.800	);(	0.000	,	 5.600	)**\dir{-},
(	1.700	,	6.300	);(	8.000	,	-0.000	)**\dir{-},
(	8.000	,	-6.500	);(	8.000	,	-1.500	)**\dir{.},
(	8.000	,	-8.000	);(	16.000	,	0.000	)**\dir{-},
(      16.000	,	1.500	);(	16.000	,	 6.500	)**\dir{.},
{	"h1"	\ar@/	_	2.5	mm	/@{-}	_	(	0.5	){	i	}	"hh"	},
{	"d1"	\ar@/	_	2.5	mm	/@{-}	_	(	0.5	){	j	}	"dd"	},
\end{xy}
,\qquad
w_d=W(v_{a-d,d}).
$$
We first show
\begin{equation}\label{hodai_lim1}
W \left(~
\begin{xy}
(	0	,	- 10.00	)	*+{}	                                                            *{\bullet};
(	0	,	-  4.00	)	*++{\bf l}	                                                    *\frm{o};
(	0	,	   2.00	)	*+{}	                                                            *{\bullet}      ="a1";	
(	0	,         10.00	)	*+{}						                 *{\bullet}	="aa";
(	0.000	,	-  9.400	);(	0.000	,	- 6.400	)**\dir{-},
(	0.000	,	-  1.500	);(	0.000	,	   1.500	)**\dir{-},
(	0.000	,	  3.500	);(	0.000	,	 8.500	)**\dir{.},
{	"a1"	\ar@/	^	2	mm	/@{-}	^	(	0.5	){	a	}	"aa"	},
\end{xy}
~\right)
-\sum_{d=1}^{a}
(-1)^{a-d}
W (v_{a-d,0})
W \left(~  
\begin{xy}
(	0.000	,        4.000	)	*+{}	                                                            *{\bullet}      ="h1";	
(	0.000	,       -  4.000	)	*+{}						                 *{\bullet}	="hh";
(	0.000	,	-3.200	);(	0.000	,	 2.500	)**\dir{.},
{	"h1"	\ar@/	_	2.5	mm	/@{-}	_	(	0.5	){	d	}	"hh"	},
\end{xy}
~\right)
=
(-1)^a
W (v_{a,0})
.
\end{equation}
Since $W$ is a ring homomorphism satisfying the property \eqref{xab}, we have
$$
W (v_{a-d,0})
W \left(~  
\begin{xy}
(	0.000	,        4.000	)	*+{}	                                                            *{\bullet}      ="h1";	
(	0.000	,       -  4.000	)	*+{}						                 *{\bullet}	="hh";
(	0.000	,	-3.200	);(	0.000	,	 2.500	)**\dir{.},
{	"h1"	\ar@/	_	2.5	mm	/@{-}	_	(	0.5	){	d	}	"hh"	},
\end{xy}
~\right)
=
W \left(~
v_{a-d,0}
\amalg
\begin{xy}
(	0.000	,        4.000	)	*+{}	                                                            *{\bullet}      ="h1";	
(	0.000	,       -  4.000	)	*+{}						                 *{\bullet}	="hh";
(	0.000	,	-3.200	);(	0.000	,	 2.500	)**\dir{.},
{	"h1"	\ar@/	_	2.5	mm	/@{-}	_	(	0.5	){	d	}	"hh"	},
\end{xy}
~\right)
=w_d+w_{d+1}. 
$$
Hence we obtain \eqref{hodai_lim1} by summing them up with signs. 
Applying the map  ${\rm Li}(\ \cdot\ ;z)$ to \eqref{hodai_lim1} and taking the limit as $z\to1$, we obtain
\begin{equation}
\label{hodai_lim2}
\lim_{z\to 1}\left\{
{\rm Li}({\bf l},\{1\}^a;z)
-
\sum_{d=1}^{a}
(-1)^{a-d}
{\rm Li}(W(v_{a-d,0});z)
{\rm Li}(\{1\}^d;z)
\right\}
=
(-1)^a
\zeta(W(v_{a,0}))
.
\end{equation}
On the other hand, 
since for any admissible index ${\bf k}$, there exists a positive integer $J$ such that
$${\rm Li}({\bf k};z)-\zeta({\bf k})=O((1-z)(\log((1-z))^J))$$
 ({\cite[p. 311]{ikz}}), we see that 
\begin{align}\label{hodai_lim3}
\lim_{z\to 1}\left| {\rm Li}(W(v_{a-d,0});z)-\zeta(W(v_{a-d,0})) \right| |{\rm Li}(\{1\}^d;z)|=0.
\end{align}
Combining \eqref{hodai_lim2} with \eqref{hodai_lim3} and using
$$
\zeta(W (v_{a-d,0}))=
\xi({\bf l};a-d+1)
=
\sum_{\substack{{\rm wt}({\bf e})+d=a \\ {\rm dep}({\bf e})=n }}b({\bf l}^\dagger;{\bf e}) \zeta({\bf l}^\dagger+{\bf e})
$$
for $0\leq d \leq a$, which follows from \eqref{xi_poset} and Lemma \ref{kt_fcn}, we obtain the desired result. 
\end{proof}

\section{Proofs of Main Theorems}

\subsection{Proof of Theorem \ref{mainThm2}}

We now prove Theorem \ref{mainThm2}. 

\begin{proof}[Proof of Theorem \ref{mainThm2}]
We proceed by induction on the weight of ${\bf k}=(\{1\}^{a_1-1},b_1+1,\ldots,\{1\}^{a_n-1},b_n+1)$. 
When ${\rm wt}({\bf k})={\rm dep}({\bf k})$ i.e, $n=1,b_1=0$, it is obvious. 
Suppose the weight ${\rm wt}({\bf k})$ of ${\bf k}$ is grater than the depth ${\rm dep}({\bf k})$ and assume the statement holds for any index of weight ${\rm wt}({\bf k})-1$. 

First assume that ${\bf k}$ is admissible i.e., $b_n\geq 1$. 
In this case, the proof is given by the same argument in the one of {\cite[Lemma 3.5]{kt_fcn}} and hence omitted. 

We assume that ${\bf k}$ is not admissible, i.e., $b_n=0$. 
Since 
$$
{\bf k}_-=\begin{cases}
({\bf k}_{n-1},\{1\}^{a_n-2},1) & \mbox{if }a_n\geq2, \\
{\bf k}_{n-1} & \mbox{if }a_n=1,
\end{cases}
$$
it follows from induction hypothesis and duality formula for multiple zeta values that 
\begin{align*}
&{\rm Li}({\bf k}_{-};1-z)  \nonumber\\
& = -\sum_{l=1}^{n-1}  \sum_{j=0}^{b_l-2} (-1)^{j+{\rm wt}({\bf k}^l)-1} 
\zeta({\bf k}_{l-1},\{1\}^{a_l-1},b_l-j){\rm Li}((j+1,{\bf k}^l_{n-1},\{1\}^{a_n-1})^\vee;z) \\
&\quad +\sum_{l=1}^{n-1} (-1)^{b_l+{\rm wt}({\bf k}^l)-1} 
 \sum_{d= 0}^{a_l} 
 \sum_{\substack{{\rm wt}({\bf e}_1)+{\rm wt}({\bf e}_2)+d=a_l \\ {\rm dep}({\bf e}_1)=n_1,\ {\rm dep}({\bf e}_2)=n_2}}
  (-1)^{{\rm wt}({\bf e}_1)} b(({\bf k}_{l-1})^\dagger;{\bf e}_1) \\
&\quad\quad\times   
\zeta(({\bf k}_{l-1})^\dagger+{\bf e}_1) b((b_l,{\bf k}^l_{n-1},\{1\}^{a_n-1})^\vee;{\bf e}_2){\rm Li}(\{1\}^{d};1-z)
  {\rm Li}((b_l,{\bf k}^l_{n-1},\{1\}^{a_n-1})^\vee+{\bf e}_2;z) \\
&\quad + \sum_{d= 0}^{a_n-1} 
 \sum_{\substack{{\rm wt}({\bf e}_1)+d=a_n-1 \\ {\rm dep}({\bf e}_1)=n_1}}
  (-1)^{{\rm wt}({\bf e}_1)} b(({\bf k}_{n-1})^\dagger;{\bf e}_1)    
\zeta(({\bf k}_{n-1})^\dagger+{\bf e}_1) {\rm Li}(\{1\}^{d};1-z) .
\end{align*}
We have
\begin{align}
&\frac{d}{dz}{\rm Li}({\bf k};1-z)
=-\frac{1}{z}{\rm Li}({\bf k}_{n-1},\{1\}^{a_n-1};1-z)  \nonumber\\
&=-\frac{1}{z} \left\{ \vphantom{\sum_{a}^{b}}\right. 
 -\sum_{l=1}^{n-1}  \sum_{j=0}^{b_l-2} (-1)^{j+{\rm wt}({\bf k}^l)-1} 
\zeta({\bf k}_{l-1},\{1\}^{a_l-1},b_l-j){\rm Li}((j+1,{\bf k}^l_{n-1},\{1\}^{a_n-1})^\vee;z) \nonumber\\
&\quad +\sum_{l=1}^{n-1} (-1)^{b_l+{\rm wt}({\bf k}^l)-1} 
 \sum_{d= 0}^{a_l} 
 \sum_{\substack{{\rm wt}({\bf e}_1)+{\rm wt}({\bf e}_2)+d=a_l \\ {\rm dep}({\bf e}_1)=n_1,\ {\rm dep}({\bf e}_2)=n_2}}
  (-1)^{{\rm wt}({\bf e}_1)} b(({\bf k}_{l-1})^\dagger;{\bf e}_1) \nonumber\\
&\quad\quad\times   
\zeta(({\bf k}_{l-1})^\dagger+{\bf e}_1) b((b_l,{\bf k}^l_{n-1},\{1\}^{a_n-1})^\vee;{\bf e}_2){\rm Li}(\{1\}^{d};1-z)
  {\rm Li}((b_l,{\bf k}^l_{n-1},\{1\}^{a_n-1})^\vee+{\bf e}_2;z) \nonumber\\
&\quad + \sum_{d= 0}^{a_n-1} 
 \sum_{\substack{{\rm wt}({\bf e}_1)+d=a_n-1 \\ {\rm dep}({\bf e}_1)=n_1}}
  (-1)^{{\rm wt}({\bf e}_1)} b(({\bf k}_{n-1})^\dagger;{\bf e}_1)    
\zeta(({\bf k}_{n-1})^\dagger+{\bf e}_1) {\rm Li}(\{1\}^{d};1-z)  \left.\vphantom{\sum_{d}^{a}}\right\} \nonumber
\end{align}
\begin{align}
&=-\sum_{l=1}^{n-1}  \sum_{j=0}^{b_l-2} (-1)^{j+{\rm wt}({\bf k}^l)} 
\zeta({\bf k}_{l-1},\{1\}^{a_l-1},b_l-j)\frac{d}{dz}{\rm Li}({(j+1,{\bf k}^l)^\vee};z) \nonumber\\
&\quad +\frac{1}{z} \sum_{l=1}^{n-1} (-1)^{b_l+{\rm wt}({\bf k}^l)} 
 \sum_{d= 0}^{a_l} 
 \sum_{\substack{{\rm wt}({\bf e}_1)+{\rm wt}({\bf e}_2)+d=a_l \\ {\rm dep}({\bf e}_1)=n_1,\ {\rm dep}({\bf e}_2)=n_2}}
  (-1)^{{\rm wt}({\bf e}_1)} b(({\bf k}_{l-1})^\dagger;{\bf e}_1) \nonumber\\
&\quad\quad\times   
\zeta(({\bf k}_{l-1})^\dagger+{\bf e}_1) b((b_l,{\bf k}^l_{n-1},\{1\}^{a_n-1})^\vee;{\bf e}_2){\rm Li}(\{1\}^{d};1-z)
  {\rm Li}((b_l,{\bf k}^l_{n-1},\{1\}^{a_n-1})^\vee+{\bf e}_2;z) \nonumber\\
\label{thm1proof1}
&\quad +  \sum_{d= 1}^{a_n} 
 \sum_{\substack{{\rm wt}({\bf e}_1)+d=a_n \\ {\rm dep}({\bf e}_1)=n_1}}
  (-1)^{{\rm wt}({\bf e}_1)} b(({\bf k}_{n-1})^\dagger;{\bf e}_1)    
\zeta(({\bf k}_{n-1})^\dagger+{\bf e}_1) \frac{d}{dz} {\rm Li}(\{1\}^{d};1-z).
\end{align}
by using Lemma \ref{hodai2.3} in the first and the last equality. 
In the last equality, we also use the fact 
$$
((j+1,{\bf k}_{n-1}^l,\{1\}^{a_n-1})^\vee)_+
=(j+1,{\bf k}_{n-1}^l,\{1\}^{a_n}) 
=(j+1,{\bf k}^l).
$$
We compute the second term of \eqref{thm1proof1}. 
(We recall that $n_1={\rm dep}(({\bf k}_{l-1})^\dagger)$ and $n_2={\rm dep}((b_l,{\bf k}^l)^\vee)$. ) 
Denoting the last component of ${\bf e}_2$ by $e_2$ and considering two cases when $e_2=0$ and $e_2\geq 1$, we find that
\begin{align*}
& \sum_{d= 0}^{a_l} 
 \sum_{\substack{{\rm wt}({\bf e}_1)+{\rm wt}({\bf e}_2)+d=a_l \\ {\rm dep}({\bf e}_1)=n_1,\ {\rm dep}({\bf e}_2)=n_2}}
  (-1)^{{\rm wt}({\bf e}_1)} b(({\bf k}_{l-1})^\dagger;{\bf e}_1) \zeta(({\bf k}_{l-1})^\dagger+{\bf e}_1)\\
&\quad\times   
 b((b_l,{\bf k}^l_{n-1},\{1\}^{a_n-1})^\vee;{\bf e}_2){\rm Li}(\{1\}^{d};1-z)
  {\rm Li}((b_l,{\bf k}^l_{n-1},\{1\}^{a_n-1})^\vee+{\bf e}_2;z) \\
&= 
 \sum_{d= 0}^{a_l} 
 \sum_{\substack{{\rm wt}({\bf e}_1)+{\rm wt}({\bf e}_2)+d=a_l \\ {\rm dep}({\bf e}_1)=n_1,\ {\rm dep}({\bf e}_2)=n_2 \\ e_2=0}}
  (-1)^{{\rm wt}({\bf e}_1)} b(({\bf k}_{l-1})^\dagger;{\bf e}_1) \zeta(({\bf k}_{l-1})^\dagger+{\bf e}_1) \\
&\quad\quad\times   
b((b_l,{\bf k}^l)^\vee;{\bf e}_2){\rm Li}({\{1\}^{d}};1-z)
  {\rm Li}({(b_l,{\bf k}^l_{n-1},\{1\}^{a_n-1})^\vee+{\bf e}_2};z) \\
&\quad + 
 \sum_{d= 0}^{a_l} 
 \sum_{\substack{{\rm wt}({\bf e}_1)+{\rm wt}({\bf e}_2)+d+1=a_l \\ {\rm dep}({\bf e}_1)=n_1,\ {\rm dep}({\bf e}_2)=n_2}}
  (-1)^{{\rm wt}({\bf e}_1)} b(({\bf k}_{l-1})^\dagger;{\bf e}_1) \zeta(({\bf k}_{l-1})^\dagger+{\bf e}_1) \\
&\quad\quad\times   
( b((b_l,{\bf k}^l)^\vee;({\bf e}_2)_+) - b((b_l,{\bf k}^l)^\vee;{\bf e}_2)){\rm Li}({\{1\}^{d}};1-z)
  {\rm Li}({(b_l,{\bf k}^l_{n-1},\{1\}^{a_n-1})^\vee+({\bf e}_2)_+};z) \\
&= 
 \sum_{d= 0}^{a_l} 
 \sum_{\substack{{\rm wt}({\bf e}_1)+{\rm wt}({\bf e}_2)+d=a_l \\ {\rm dep}({\bf e}_1)=n_1,\ {\rm dep}({\bf e}_2)=n_2}}
  (-1)^{{\rm wt}({\bf e}_1)} b(({\bf k}_{l-1})^\dagger;{\bf e}_1) \zeta(({\bf k}_{l-1})^\dagger+{\bf e}_1) \\
&\quad\quad\times   
b((b_l,{\bf k}^l)^\vee;{\bf e}_2){\rm Li}({\{1\}^{d}};1-z)
  {\rm Li}({(b_l,{\bf k}^l_{n-1},\{1\}^{a_n-1})^\vee+{\bf e}_2};z) \\
&\quad - 
 \sum_{d= 1}^{a_l+1} 
 \sum_{\substack{{\rm wt}({\bf e}_1)+{\rm wt}({\bf e}_2)+d=a_l \\ {\rm dep}({\bf e}_1)=n_1,\ {\rm dep}({\bf e}_2)=n_2}}
  (-1)^{{\rm wt}({\bf e}_1)} b(({\bf k}_{l-1})^\dagger;{\bf e}_1) \zeta(({\bf k}_{l-1})^\dagger+{\bf e}_1) \\
&\quad\quad\times   
b((b_l,{\bf k}^l)^\vee;{\bf e}_2){\rm Li}({\{1\}^{d-1}};1-z)
  {\rm Li}({(b_l,{\bf k}^l)^\vee+{\bf e}_2};z).
  \end{align*}
Since 
\begin{align*}
&\frac{d}{dz}\left\{{\rm Li}({\{1\}^{d}};1-z) {\rm Li}((b_l,{\bf k}^l)^\vee+{\bf e}_2;z) \right\}\\
&=\frac{1}{z}\left\{{\rm Li}({\{1\}^{d}};1-z) {\rm Li}((b_l,{\bf k}^l_{n-1},\{1\}^{a_n-1})^\vee+{\bf e}_2;z)
+{\rm Li}({\{1\}^{d-1}};1-z) {\rm Li}((b_l,{\bf k}^l)^\vee+{\bf e}_2;z)\right\}, 
\end{align*}
the second term of \eqref{thm1proof1} coincides with 
  \begin{align*}
&
\sum_{l=1}^{n-1} (-1)^{b_l+{\rm wt}({\bf k}^l)} 
 \sum_{d= 0}^{a_l} 
 \sum_{\substack{{\rm wt}({\bf e}_1)+{\rm wt}({\bf e}_2)+d=a_l \\ {\rm dep}({\bf e}_1)=n_1,\ {\rm dep}({\bf e}_2)=n_2}}
  (-1)^{{\rm wt}({\bf e}_1)} b(({\bf k}_{l-1})^\dagger;{\bf e}_1) \zeta(({\bf k}_{l-1})^\dagger+{\bf e}_1) \\
&\times   
b((b_l,{\bf k}^l)^\vee;{\bf e}_2)\frac{d}{dz}\left\{{\rm Li}({\{1\}^{d}};1-z)
  {\rm Li}((b_l,{\bf k}^l)^\vee+{\bf e}_2;z) \right\}.    
\end{align*}
We therefore conclude
\begin{align*}
&{\rm Li}({\bf k};1-z) \\
&= -\sum_{l=1}^{n-1}  \sum_{j=0}^{b_l-2} (-1)^{j+{\rm wt}({\bf k}^l)} 
\zeta({\bf k}_{l-1},\{1\}^{a_l-1},b_l-j){\rm Li}({(j+1,{\bf k}^l)^\vee};z) \\
&\quad+
\sum_{l=1}^{n-1} (-1)^{b_l+{\rm wt}({\bf k}^l)} 
 \sum_{d= 0}^{a_l} 
 \sum_{\substack{{\rm wt}({\bf e}_1)+{\rm wt}({\bf e}_2)+d=a_l \\ {\rm dep}({\bf e}_1)=n_1,\ {\rm dep}({\bf e}_2)=n_2}}
  (-1)^{{\rm wt}({\bf e}_1)} b(({\bf k}_{l-1})^\dagger;{\bf e}_1) \zeta(({\bf k}_{l-1})^\dagger+{\bf e}_1) \\
&\quad\quad\times   
b((b_l,{\bf k}^l)^\vee;{\bf e}_2){\rm Li}({\{1\}^{d}};1-z)
  {\rm Li}((b_l,{\bf k}^l)^\vee+{\bf e}_2;z)    \\
&\quad + \sum_{d= 1}^{a_n} 
 \sum_{\substack{{\rm wt}({\bf e}_1)+d=a_n \\ {\rm dep}({\bf e}_1)=n_1}}
  (-1)^{{\rm wt}({\bf e}_1)} b(({\bf k}_{n-1})^\dagger;{\bf e}_1)    
\zeta(({\bf k}_{n-1})^\dagger+{\bf e}_1) {\rm Li}(\{1\}^d;1-z)  +C
\end{align*}
for some constant $C$. 
We find 
\begin{align*}
C=\sum_{\substack{{\rm wt}({\bf e}_1)=a_n \\ {\rm dep}({\bf e}_1)=n_1 }} (-1)^{{\rm wt}({\bf e}_1)}
b(({\bf k}_{n-1})^\dagger;{\bf e}_1) \zeta(({\bf k}_{n-1})^\dagger+{\bf e}_1)
\end{align*}
by setting $z\rightarrow0$ and applying Lemma \ref{hodai_lim}.  
This completes the proof.  
\end{proof}


\begin{proof}[Proof of Theorem \ref{mainThm1}]
By setting $z=e^{-t}$ in Theorem \ref{mainThm2} and using ${\rm Li}(\{1\}^d;1-e^t)={t^d}/{d!}$ ({\cite[Lemma 1 (ii)]{ak_fcn}}), 
we have
\begin{align*}
& {\rm Li}({\bf k};1-e^{-t}) \\
%
&=(1-\delta_{0,b_n})\zeta({\bf k})
 -\sum_{l=1}^{n}  \sum_{j=0}^{b_l-2} (-1)^{j+{\rm wt}({\bf k}^l)} 
\zeta({\bf k}_{l-1},\{1\}^{a_l-1},b_l-j){\rm Li}((j+1,{\bf k}^l)^\vee;e^{-t}) \\
&\quad +\sum_{l=1}^{n} (-1)^{b_l+{\rm wt}({\bf k}^l)} 
 \sum_{d= 0}^{a_l} 
 \sum_{\substack{{\rm wt}({\bf e}_1)+{\rm wt}({\bf e}_2)+d=a_l \\ {\rm dep}({\bf e}_1)=n_1,\ {\rm dep}({\bf e}_2)=n_2}}
  (-1)^{{\rm wt}({\bf e}_1)} b(({\bf k}_{l-1})^\dagger;{\bf e}_1) \\
&\quad\quad\times   
\zeta(({\bf k}_{l-1})^\dagger+{\bf e}_1) b((b_l,{\bf k}^l)^\vee;{\bf e}_2)\frac{t^d}{d!}
  {\rm Li}((b_l,{\bf k}^l)^\vee+{\bf e}_2;e^{-t}).
\end{align*}
Substituting this into the definition \eqref{intrep_ak} of  $\xi({\bf k};s)$ and using the one \eqref{intrep_ez} of $\zeta({\bf k};s)$, we obtain the theorem. 
\end{proof}

\section{Proofs of Corollary \ref{mainCor}}
We give two proofs of Corollary \ref{mainCor}. 
One is to apply Theorem \ref{mainThm2} and the other is to use $2$-poset. 

\subsection{An application of Theorem \ref{mainThm1}}

\begin{proof}[Proof of Corollary \ref{mainCor}]
Substituting $s= m+1$, we have
\begin{align*}
&\xi({\bf k};m+1) \\
&=(1-\delta_{0,b_n})\zeta({\bf k})\zeta(m+1) \\
&\quad -\sum_{l=1}^{n} \sum_{j=0}^{b_l-2} (-1)^{j+{\rm wt}({\bf k}^l)} 
\zeta({\bf k}_{l-1},\{1\}^{a_l-1},b_l-j) \zeta((j+1,{\bf k}^l)^\vee,m+1) \\
&\quad +\sum_{l=1}^{n} (-1)^{b_{l}+{\rm wt}({\bf k}^{l})}
 \sum_{d= 0}^{a_{l}}
 \sum_{\substack{{\rm wt}({\bf e}_1)+{\rm wt}({\bf e}_2)+d=a_{l} \\ {\rm dep}({\bf e}_1)=n_1,\ {\rm dep}({\bf e}_2)=n_2}} (-1)^{{\rm wt}({\bf e}_1)} b(({\bf k}_{l-1})^\dagger;{\bf e}_1) \zeta(({\bf k}_{l-1})^\dagger+{\bf e}_1)\\
&\qquad\times    b((b_{l},{\bf k}^{l})^\vee;{\bf e}_2)  \binom{m+d}{d} \zeta((b_{l},{\bf k}^{l})^\vee+{\bf e}_2,m+d+1). 
\end{align*}
Applying Lemma \ref{kt_fcn} to the third term of the right hand side, we see that
\begin{align*}
&\sum_{d= 0}^{a_{l}}
 \sum_{\substack{{\rm wt}({\bf e}_1)+{\rm wt}({\bf e}_2)+d=a_{l} \\ {\rm dep}({\bf e}_1)=n_1,\ {\rm dep}({\bf e}_2)=n_2}} (-1)^{{\rm wt}({\bf e}_1)} b(({\bf k}_{l-1})^\dagger;{\bf e}_1) \zeta(({\bf k}_{l-1})^\dagger+{\bf e}_1)\\
&\quad\times    b((b_{l},{\bf k}^{l})^\vee;{\bf e}_2)  \binom{m+d}{d} \zeta((b_{l},{\bf k}^{l})^\vee+{\bf e}_2,m+d+1)\\
&=  \sum_{\substack{{\rm wt}({\bf e}_1)+{\rm wt}({\bf e}_2)=a_{l} \\ {\rm dep}({\bf e}_1)=n_1,\ {\rm dep}({\bf e}_2)=n_2+1}} (-1)^{{\rm wt}({\bf e}_1)} b(({\bf k}_{l-1})^\dagger;{\bf e}_1) \zeta(({\bf k}_{l-1})^\dagger+{\bf e}_1) \\
&\quad\times   
 b(((b_{l},{\bf k}^{l})^\vee,m+1);{\bf e}_2)   \zeta(((b_{l},{\bf k}^{l})^\vee,m+1)+{\bf e}_2) \\
&= \sum_{j=0}^{a_{l}} \sum_{\substack{{\rm wt}({\bf e}_1)=j \\ {\rm dep}({\bf e}_1)=n_1}}
  (-1)^{{\rm wt}({\bf e}_1)} b(({\bf k}_{l-1})^\dagger;{\bf e}_1) \zeta(({\bf k}_{l-1})^\dagger+{\bf e}_1) \xi (((b_{l},{\bf k}^{l})^\vee,m+1)^\dagger)_-;a_{l}-j+1) \\
&= \sum_{j=0}^{a_{l}} (-1)^j \xi(({\bf k}_{l-1})_-;j+1) 
 \xi (\{1\}^{m-1},\overleftarrow{(b_{l},{\bf k}^{l})_+};a_{l}-j+1).
 \end{align*}
Hence we obtain
\begin{align*}
&\xi({\bf k};m+1) \\
&=(1-\delta_{0,b_n})\zeta({\bf k})\zeta(m+1)
 -\sum_{l=1}^{n} \sum_{j=0}^{b_l-2} (-1)^{j+{\rm wt}({\bf k}^l)} 
\zeta({\bf k}_{l-1},\{1\}^{a_l-1},b_l-j) \zeta((j+1,{\bf k}^l)^\vee,m+1) \\
&\quad +\sum_{l=1}^{n} (-1)^{b_{l}+{\rm wt}({\bf k}^{l})} 
\sum_{j=0}^{a_{l}} (-1)^j \xi(({\bf k}_{l-1})_-;j+1)  \xi (\{1\}^{m-1},\overleftarrow{(b_{l},{\bf k}^{l})_+};a_{l}-j+1).
\end{align*}
\end{proof}

\subsection{Combinatorial proof of Corollary \ref{mainCor}}

We can give combinatorial proof of Corollary \ref{mainCor} by using $2$-poset and their associated integrals. 
Let $a_1,b_1,\ldots,a_{n-1},b_{n-1},a_n \geq1,\ b_n \geq 0$. 
For a positive index ${\bf k}=(\{1\}^{a_1-1},b_1+1,\ldots,\{1\}^{a_n-1},b_n+1)$, we write 
$\begin{xy}
(	0	,	- 3.500	)	*+{}	                                                            *{\odot};
(	7	,	    3.500	)	*+[F]{\overleftarrow{\bf k}}	                                                    ;
(	0.900	,	-  2.600	);(	4.000	,	 0.500	)**\dir{-},
\end{xy} $ 
for the diagram
$$
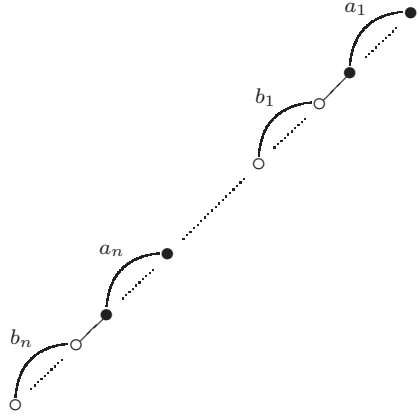
\begin{xy}
(8,-22)			*+	{}						*{\circ}	="a1_l";
(16,-14)			*+	{}						*{\circ}	="a1_r";
{(8,-20.8) \ar @/^2mm/ @{-}^(0.6){b_{n}} (14.8,-13.8)}, 
(20,-10)			*+	{}						*{\bullet}	="b1_l";
(28,-2)			*+	{}						*{\bullet}	="b1_r";
{(20,-8.8) \ar @/^2mm/ @{-}^(0.6){a_{n}} (26.8,-1.8)}, 
(40,10)			*+	{}						*{\circ}	="as_l";
(48,18)			*+	{}						*{\circ}	="as_r";
{(40,11.2) \ar @/^2mm/ @{-}^(0.6){b_{1}} (46.8,18.2)},
(52,22)			*+	{}						*{\bullet}	="bs_l";
(60,30)			*+	{}						*{\bullet}	="bs_r";
{(52,23.2) \ar @/^2mm/ @{-}^(0.6){a_{1}} (58.8,30.2)},
(10.08,-19.92);(14.08,-15.92) **\dir{.},
(16.58,-13.42);(19.58,-10.42) **\dir{-},
(22.08,-7.92);(26.08,-3.92) **\dir{.},
(30.08,0.08);(38.08,8.08) **\dir{.},
(42.08,12.08);(46.08,16.08) **\dir{.}, 
(48.58,18.58);(51.58,21.58) **\dir{-},
(54.08,24.08);(58.08,28.08) **\dir{.},
\end{xy}
$$
obtained by turning 
$\begin{xy}
(0,-2.8)			*+	{}						*{\bullet};
(5.6,2.8)			*++	{\bf k}						*\frm{o};
(	0.5800	,	-2.220	);(	3.580	,	0.780	)**\dir{-},
\end{xy}
$ upside down. 


\begin{hodai}\label{idou}
Let ${\bf k}, {\bf l}$ be positive indices. 
\begin{description}
\item[(1)] For any non-negative integer $b$, we have
$$
W \left(~
\begin{xy}
(	0	,	- 14.00	)	*+{}	                                                            *{\bullet};
(	0	,	-  8.00	)	*++{\bf k}	                                                    *\frm{o};
(	0	,	-  2.00	)	*+{}	                                                            *{\circ}      ="b1";	
(	0	,         6.00	)	*+{}						                 *{\circ}	="bb";
(	8	,	 14.00  	)	        *++{\bf l}	                                                  *\frm{o}     ="bfl";	
(	8	,        8.00	)	*+{}	*{\bullet};
(	0.000	,	-  13.400	);(	0.000	,	- 10.800	)**\dir{-},
(	0.000	,	-  5.100	);(	0.000	,	 - 2.500	)**\dir{-},
(	0.000	,	- 0.500	);(	0.000	,	 4.500	)**\dir{.},
(	0.500	,	6.500	);(	6.300	,	12.300	)**\dir{-},
(	8.0000	,	  8.600	);(	8.000	,	 11.600	)**\dir{-},
{	"b1"	\ar@/	^	2	mm	/@{-}	^	(	0.5	){	b	}	"bb"	},
\end{xy}
~\right)
=
\sum_{j=0}^{b-1}
(-1)^j
W
\left(~
\begin{xy}
(	0	,	- 10.00	)	*+{}	                                                            *{\bullet};
(	0	,	-  4.00	)	*++{\bf k}	                                                    *\frm{o};
(	0	,	  2.00	)	*+{}	                                                            *{\circ}      ="b1";	
(	0	,         10.00	)	*+{}						                 *{\circ}	="bb";
(	0.000	,	-  9.500	);(	0.000	,	- 6.800	)**\dir{-},
(	0.000	,	-  1.200	);(	0.000	,	 1.500	)**\dir{-},
(	0.000	,	3.500	);(	0.000	,	 8.500	)**\dir{.},
{	"b1"	\ar@/	^	2	mm	/@{-}	^	(	0.5	){	b-j	}	"bb"	},
\end{xy}
~\right)
W \left(~
\begin{xy}
(	0	,	- 10.00	)	*+{}	                                                            *{\bullet};
(	0	,	-  4.00	)	*++{\bf l}	                                                    *\frm{o};
(	0	,	  2.00	)	*+{}	                                                            *{\circ}      ="b1";	
(	0	,         10.00	)	*+{}						                 *{\circ}	="bb";
(	0.000	,	-  9.500	);(	0.000	,	- 6.400	)**\dir{-},
(	0.000	,	-  1.500	);(	0.000	,	 1.500	)**\dir{-},
(	0.000	,	3.500	);(	0.000	,	 8.500	)**\dir{.},
{	"b1"	\ar@/	^	2	mm	/@{-}	^	(	0.5	){	j	}	"bb"	},
\end{xy}
~\right)
+(-1)^b
W
\left(~
\begin{xy}
(	0	,	   -4.00	)	*+{}	*{\bullet};
(	0	,	 2.00  	)	        *++{\bf k}	                                                  *\frm{o};	
(	8.000	,	- 10.00	)	*+{}	                                                            *{\bullet};
(	8.000	,	-  4.00	)	*++{\bf l}	                                                    *\frm{o};
(	8.000	,	   2.00	)	*+{}	                                                            *{\circ}      ="b1";	
(	8.000	,         10.00	)	*+{}						                 *{\circ}	="bb";
(	0.000	,	  -3.200	);(	0.000	,	 -0.800	)**\dir{-},
(	2.10	,	4.100	);(	7.50	,	9.50	)**\dir{-},
(	8.000	,	-  9.300	);(	8.000	,	- 6.500	)**\dir{-},
(	8.000	,	-  1.400	);(	8.000	,	  1.400	)**\dir{-},
(	8.000	,	3.500	);(	8.000	,	 8.500	)**\dir{.},
{	"b1"	\ar@/	_	2.5	mm	/@{-}	_	(	0.5	){	b	}	"bb"	},
\end{xy}
~\right).
$$
\item[(2)] For non-negative integer $a$, we have
$$
W\left(~
\begin{xy}
(	0	,	- 14.00	)	*+{}	                                                            *{\bullet};
(	0	,	-  8.00	)	*++{\bf k}	                                                    *\frm{o};
(	0	,	-  2.00	)	*+{}	                                                            *{\bullet}      ="a1";	
(	0	,         6.00	)	*+{}						                 *{\bullet}	="aa";
(	8	,	 14.00  	)	        *++{\bf l}	                                                  *\frm{o}     ="bfl";	
(	8	,        8.00	)	*+{}	*{\bullet};
(	0.000	,	-  13.400	);(	0.000	,	- 10.800	)**\dir{-},
(	0.000	,	-  5.100	);(	0.000	,	 - 2.500	)**\dir{-},
(	0.000	,	- 0.500	);(	0.000	,	 4.500	)**\dir{.},
(	0.500	,	6.500	);(	6.300	,	12.300	)**\dir{-},
(	8.0000	,	  8.600	);(	8.000	,	 11.600	)**\dir{-},
{	"a1"	\ar@/	^	2	mm	/@{-}	^	(	0.5	){	a	}	"aa"	},
\end{xy}
~\right)
=
\sum_{j=0}^{a}
(-1)^j
W
 \left(~
\begin{xy}
(	0	,	2.00	)	*+{}	                                                            *{\bullet};
(	0	,	  8.00	)	*++{\bf k}	                                                    *\frm{o};
(	8.000	,        0.000	)	*+{}	                                                            *{\bullet}      ="h1";	
(	8.000	,       -  8.000	)	*+{}						                 *{\bullet}	="hh";
(	0.000	,        2.800	);(	0.000	,	 5.200	)**\dir{-},
(	2.000	,	6.000	);(	8.000	,	-0.000	)**\dir{-},
(	8.000	,	-6.500	);(	8.000	,	-1.500	)**\dir{.},
{	"h1"	\ar@/	_	2.5	mm	/@{-}	_	(	0.5	){	j	}	"hh"	},
\end{xy}
~\right)
W \left(~  
\begin{xy}
(	0.000	,        0.000	)	*+{}	                                                            *{\bullet}      ="h1";	
(	0.000	,       -  8.000	)	*+{}						                 *{\bullet}	="hh";
(	8.000	,	2.00	)	*+{}	                                                            *{\bullet};
(	8.000	,	  8.00	)	*++{\bf l}	                                                    *\frm{o};
(	0.000	,	-7.200	);(	0.000	,	-1.500	)**\dir{.},
(	0.600	,        0.600	);(	6.200	,	 6.200	)**\dir{-},
(	8.000	,	2.600	);(	8.000	,	5.600	)**\dir{-},
{	"h1"	\ar@/	_	2.5	mm	/@{-}	_	(	0.5	){	a-j	}	"hh"	},
\end{xy}
~\right)
+(-1)^{a+1}
W
\left(~
\begin{xy}
(	0	,	   8.000	)	*+{}	*{\bullet};
(	0	,	 14.000  	)	        *++{\bf k}	                                                  *\frm{o};
(	8.000	,	- 14.00	)	*+{}	                                                            *{\bullet};
(	8.000	,	-  8.00	)	*++{\bf l}	                                                    *\frm{o};
(	8.000	,	 -2.000	)	*+{}	                                                            *{\bullet}      ="b1";	
(	8.000	,         6.000	)	*+{}						                 *{\bullet}	="bb";
(	0.000	,	  8.600	);(	0.000	,	 11.200	)**\dir{-},
(	2.00	,	12.000	);(	8.000	,	6.000	)**\dir{-},
(	8.000	,	-  13.300	);(	8.000	,	- 10.400	)**\dir{-},
(	8.000	,	-  5.400	);(	8.000	,	- 2.500	)**\dir{-},
(	8.000	,	-0.500	);(	8.000	,	 4.500	)**\dir{.},
{	"b1"	\ar@/	_	2.5	mm	/@{-}	_	(	0.5	){	a	}	"bb"	},
\end{xy}
~\right).
$$
\end{description}
\end{hodai}

\begin{proof}
The lemma can be shown in a similar way as \eqref{hodai_lim1}. 
\end{proof}

\begin{proof}[Combinatorial proof of Corollary \ref{mainCor}]
Let ${\bf k}=(\{1\}^{a_1-1},b_1+1,\ldots,\{1\}^{a_n-1},b_n+1)$. 
Using Lemma \ref{idou} repeatedly, we find that 
\begin{align*}
& W \left(~
\begin{xy}
(	0	,	   -6.00	)	*+{}	*{\bullet};
(	0	,	 0.00  	)	        *++{\bf k}	                                                  *\frm{o};	
(	8	,	 8.00	)	*+{}                                                        	*{\circ};
(	16.000	,	-  8.00	)	*+{}	                                                            *{\bullet}      ="m1";	
(	16.000	,         0.00	)	*+{}						                 *{\bullet}	="mm";
(	0.000	,	  -5.200	);(	0.000	,	 -2.800	)**\dir{-},
(	2.10	,	2.100	);(	7.50	,	7.50	)**\dir{-},
(	8.500	,	7.50	);(	15.450	,	0.750	)**\dir{-},
(	16.000	,	-6.500	);(	16.000	,	 -1.500	)**\dir{.},
{	"m1"	\ar@/	_	2.5	mm	/@{-}	_	(	0.5	){	m	}	"mm"	},
\end{xy}
~\right) \\
&\quad=
(1-\delta_{0,b_n})W \left(~
\begin{xy}
(	0	,	- 3.00	)	*+{}	                                                            *{\bullet};
(	0	,	  3.00	)	*++{\bf k}	                                                    *\frm{o};
(	0.000	,	-  2.500	);(	0.000	,	 0.320	)**\dir{-},
\end{xy}
~\right)
I \left(~
\begin{xy}
(	0.000	,	-  6.00	)	*+{}	                                                            *{\bullet}      ="m1";	
(	0.000	,          2.00	)	*+{}						                 *{\bullet}	="mm";
(	0	,	  6.00	)	*+{}	                                                            *{\circ}      ="b1";	
(	0.000	,	-4.500	);(	0.000	,	 0.500	)**\dir{.},
(	0.000	,	   2.500	);(	0.000	,	 5.500	)**\dir{-},
{	"m1"	\ar@/	^	2.5	mm	/@{-}	^	(	0.5	){	m	}	"mm"	},
\end{xy}
~\right) \\
&\quad-\sum_{l=0}^{h-1} (-1)^{{\rm wt}({\bf k}^{n-l})}
\sum_{j=0}^{b_{n-l}-2}(-1)^j
W \left(~
\begin{xy}
(	0	,	- 20.00	)	*+{}	                                                            *{\bullet};
(	0	,	- 10.00	)	*++{{\bf k}_{n-l-1}}	                                                    *\frm{o};
(	0	,	   0.00	)	*+{}	                                                            *{\bullet}      ="an1";	
(	0	,          8.00	)	*+{}						                 *{\bullet}	="anan";
(	0	,	  12.00	)	*+{}	                                                            *{\circ}      ="b1";	
(	0	,         20.00	)	*+{}						                 *{\circ}	="bb";
(	0.000	,	-  19.500	);(	0.000	,	- 16.800	)**\dir{-},
(	0.000	,	-  3.000	);(	0.000	,	 0.000	)**\dir{-},
(	0.000	,	  1.500	);(	0.000	,	 6.500	)**\dir{.},
(	0.000	,	   8.500	);(	0.000	,	 11.500	)**\dir{-},
(	0.000	,	13.500	);(	0.000	,	 18.500	)**\dir{.},
{	"an1"	\ar@/	^	2	mm	/@{-}	^	(	0.5	){	a_{n-l}	}	"anan"	},
{	"b1"	\ar@/	^	2	mm	/@{-}	^	(	0.5	){	b_{n-l}-j-1	}	"bb"	},
\end{xy}
~\right)
W \left(~
\begin{xy}
(	0.000	,	-  20.00	)	*+{}	                                                            *{\bullet}      ="m1";	
(	0.000	,       -  12.00	)	*+{}						                 *{\bullet}	="mm";
(	0	,	-   8.00	)	*+{}	                                                            *{\circ};
(	0	,	-   4.00	)	*+{}	                                                            *{\odot};
(	0	,	    4.00	)	*+[F]{\overleftarrow{{\bf k}^{n-l}}}	                                                  ;
(	0	,	  12.00	)	*+{}	                                                            *{\circ}      ="b1";	
(	0	,         20.00	)	*+{}						                 *{\circ}	="bb";
(	0.000	,	-18.500	);(	0.000	,	 -13.500	)**\dir{.},
(	0.000	,	-  11.500	);(	0.000	,	 -8.600	)**\dir{-},
(	0.000	,	-  7.300	);(	0.000	,	 -5.200	)**\dir{-},
(	0.000	,	-  2.800	);(	0.000	,	  0.500	)**\dir{-},
(	0.000	,	  7.500	);(	0.000	,	 11.500	)**\dir{-},
(	0.000	,	13.500	);(	0.000	,	 18.500	)**\dir{.},
{	"b1"	\ar@/	^	2	mm	/@{-}	^	(	0.5	){	j+1	}	"bb"	},
{	"m1"	\ar@/	^	2.5	mm	/@{-}	^	(	0.5	){	m	}	"mm"	},
\end{xy}
~\right) 
\end{align*}
\begin{align*}
&\quad+\sum_{l=1}^{h-1} (-1)^{b_{n-l+1}+{\rm wt}({\bf k}^{n-l+1})}
\sum_{j=0}^{a_{n-l+1}} (-1)^j
W \left(~
\begin{xy}
(	0	,	0.00	)	*+{}	                                                            *{\bullet};
(	0	,	  8.00	)	*+{{\bf k}_{n-l}}	                                                    *\frm{o};
(	8.000	,        0.000	)	*+{}	                                                            *{\bullet}      ="h1";	
(	8.000	,       -  8.000	)	*+{}						                 *{\bullet}	="hh";
(	0.000	,        0.800	);(	0.000	,	 3.800	)**\dir{-},
(	3.000	,	5.000	);(	8.000	,	-0.000	)**\dir{-},
(	8.000	,	-6.500	);(	8.000	,	-1.500	)**\dir{.},
{	"h1"	\ar@/	_	2.5	mm	/@{-}	_	(	0.5	){	j	}	"hh"	},
\end{xy}
~\right)
W \left(~  
\begin{xy}
(	0	,	  8.00	)	*+{}	                                                            *{\bullet}      ="a1";	
(	0	,         16.000	)	*+{}						                 *{\bullet}	="aa";
(	8	,	 24.000	)	*+{}                                                        	*{\circ};
(	16.0	,	 - 24.00	)	*+{}	*{\bullet}    ="m1";
(	16.0	,	 - 16.00	)	*+{}	*{\bullet}    ="mm";
(	16.0	,	 - 12.00	)	*+{}	*{\circ} ;
(	16.0	,	 - 8.00	)	*+{}	*{\odot} ;
(	16.0	,	   0.000  	)       *+[F]{\overleftarrow{{\bf k}^{n-l+1}}}	      ;	
(	16	,	  8.00	)	*+{}	                                                            *{\circ}      ="j1";
(	16	,         16.000	)	*+{}						                 *{\circ}	="jj";
(	0.000	,	9.500	);(	0.000	,	 14.500	)**\dir{.},
(	0.550	,	16.550	);(	7.550	,	23.550	)**\dir{-},
(	8.50	,	23.500	);(	15.40	,	16.60	)**\dir{-},
(	16.000	,	9.500	);(	16.000	,	 14.500	)**\dir{.},
(	16.000	,	  3.400	);(	16.000	,	  7.5000	)**\dir{-},
(	16.000	,	  -6.800	);(	16.000	,	 - 3.5000	)**\dir{-},
(	16.000	,	 - 11.200	);(	16.000	,	-  9.2000	)**\dir{-},
(	16.000	,	 - 16.200	);(	16.000	,	-  12.6000	)**\dir{-},
(	16.000	,        - 22.500	);(	16.000	,	-  17.500	)**\dir{.},
{	"a1"	\ar@/	^	2	mm	/@{-}	^	(	0.5	){	a_{n-l+1}-j	}	"aa"	},
{	"j1"	\ar@/	_	2.5	mm	/@{-}	_	(	0.5	){	b_{n-l+1}-1}	"jj"	},
{	"m1"	\ar@/	_	2.5	mm	/@{-}	_	(	0.5	){	m}	"mm"	},
\end{xy}
~\right) \\
&\quad + (-1)^{b_{n-h+1}+{\rm wt}({\bf k}^{n-h+1})}
W \left(~  
\begin{xy}
(	0	,	  - 8.00	)	*+{}	                                                            *{\bullet} ;	
(	0	,	   0.000  	)       *+{{\bf k}_{n-h}}	         *\frm{o} ;	
(	0	,	  8.00	)	*+{}	                                                            *{\bullet}      ="a1";	
(	0	,         16.000	)	*+{}						                 *{\bullet}	="aa";
(	8	,	 24.000	)	*+{}                                                        	*{\circ};
(	16	,	  8.00	)	*+{}	                                                            *{\circ}      ="j1";
(	16	,         16.000	)	*+{}						                 *{\circ}	="jj";
(	16.0	,	   0.000  	)       *+[F]{\overleftarrow{{\bf k}^{n-h+1}}}	        ;	
(	16.0	,	 - 8.00	)	*+{}	*{\odot} ;
(	16	,	 - 12.000	)	*+{}                                                        	*{\circ};
(	16.0	,	 - 16.00	)	*+{}	*{\bullet}    ="mm";
(	16.0	,	 - 24.00	)	*+{}	*{\bullet}    ="m1";
(	0.000	,	9.500	);(	0.000	,	 14.500	)**\dir{.},
(	0.000	,	- 7.200	);(	0.000	,        -4.8000	)**\dir{-},
(	0.000	,	  5.000	);(	0.000	,        8.4000	)**\dir{-},
(	0.550	,	16.550	);(	7.550	,	23.550	)**\dir{-},
(	8.50	,	23.500	);(	15.40	,	16.60	)**\dir{-},
(	16.000	,	9.500	);(	16.000	,	 14.500	)**\dir{.},
(	16.000	,	  3.400	);(	16.000	,	  7.5000	)**\dir{-},
(	16.000	,	 - 6.800	);(	16.000	,	-  3.500	)**\dir{-},
(	16.000	,	 - 11.200	);(	16.000	,	-  9.2000	)**\dir{-},
(	16.000	,	 - 15.500	);(	16.000	,	-  12.5000	)**\dir{-},
(	16.000	,        - 22.500	);(	16.000	,	-  17.500	)**\dir{.},
{	"a1"	\ar@/	^	2	mm	/@{-}	^	(	0.5	){	a_{n-h+1}	}	"aa"	},
{	"j1"	\ar@/	_	2.5	mm	/@{-}	_	(	0.5	){	b_{n-h+1}-1}	"jj"	},
{	"m1"	\ar@/	_	2.5	mm	/@{-}	_	(	0.5	){	m}	"mm"	},
\end{xy}
~\right)
\end{align*}
for each $1\leq h \leq n$. 
Applying $I$ to this relation for $h=n$, we obtain the desired result.  
\end{proof}

\section{level $2$ version}

In this section, we overview level $2$ versions of multiple zeta functions of Arakawa-Kaneko and Euler-Zagier types, and give analogues Theorem \ref{mainThm1}, \ref{mainThm2} and Corollary \ref{mainCor}.  

\subsection{Definition of $\psi({\bf k};s)$ and $T({\bf k};s)$}

Through this subsection, let ${\bf k}$ be a index. 
Kaneko and Tsumura introduced level $2$ versions of $\xi({\bf k};s)$ and $\zeta({\bf k};s)$, which are denoted by $\psi({\bf k};s)$ and $T({\bf k};s)$, respectively, as follows.
The function $\psi({\bf k};s)$ is given by
$$
\psi({\bf k};s ) =  \psi (k_1,\cdots,k_r;s) := \frac{1}{\Gamma(s)} \int_{0}^{\infty} \frac{
{\rm A}(k_1,\ldots,k_r; \tanh (t/2) )}{\sinh (t)} t^{s-1} dt
$$
for ${\rm{Re}}(s)>0$ ({\cite[Definition 5]{kt_fcn_lv2}}). 
Here, the function ${\rm A}({\bf k};z)$ is defined by
\begin{equation*}
{\rm A}({\bf k};z) =
{\rm A}(k_1,\ldots,k_r;z) := 
2^r \sum_{\substack{0<m_1<\cdots<m_r \\ m_i \equiv i \bmod 2}} \frac{z^{m_r}}{m_1^{k_1}\cdots m_r^{k_r}}
\end{equation*}
for $|z|<1$.
The function $\psi({\bf k};s)$ continues to a holomorphic function on ${\mathbb C}$. 
For ${\rm Re}(s)>0$, we have
$$
T ({\bf k};s ) =
T(k_1,\ldots,k_r;s)
:= 2^r \sum_{\substack{0<m_1<\cdots<m_{r+1} \\ m_i \equiv i \bmod 2}} \frac{1}{m_1^{k_1}\cdots m_r^{k_r}m_{r+1}^s}
$$
and $T(\varnothing;s)=(1-2^{-s})\zeta(s)$. 
For an admissible index ${\bf k}=(k_1,\ldots,k_r)$, the values 
$$T({\bf k})=T ( k_1,\ldots,k_{r-1},k_r ):=T ( k_1,\ldots,k_{r-1};k_r )$$
 are called multiple $T$-values and have duality formula $T({\bf k})=T({\bf k}^\dagger)$. 
The functions $T({\bf k};s)$ has an integral representation as
$$
T(k_1,\ldots,k_{r-1};s ) =  \frac{1}{\Gamma(s)} \int_{0}^{\infty} \frac{
{\rm A}(k_1,\ldots,k_{r-1}; e^{-t} )}{\sinh(t)} t^{s-1} dt 
$$
for ${\rm{Re}}(s)>1$ ({\cite[Proposition 2.1 (1)]{maneka_lv2}}). 
When $(k_1,k_2,\ldots,k_r)$ is an admissible index, wee see that 
$$A(k_1,\ldots,k_r;1)=T(k_1,\ldots,k_r).$$


It is known that $\psi({\bf k};s)$ and $T({\bf k};s)$ have similar properties to $\xi({\bf k};s)$ and $\zeta({\bf k};s)$. 
For instance, the following theorems give level $2$ versions of Theorem \ref{ktLem3.5} ({\cite[Lemma 3.5]{kt_fcn}}) and Theorem \ref{ktThm3.6} ({\cite[Theorem 3.6]{kt_fcn}}). 
(See also {\cite[Theorem 5.3]{kt_fcn_lv2}} for an analogue of Theorem \ref{akThm9(2)}. )

\begin{hodai}[{\cite[Lemma 2.2]{maneka_lv2}}]\label{mLem2.2}
Let ${\bf k}$ be any index. 
Then we have 
$$
{\rm A}\left({\bf k};\frac{1-z}{1+z}\right)
=\sum_{{\bf k}',d\geq0} c_{{\bf k}}({\bf k}';d) {\rm A}\left(\{1\}^d;\frac{1-z}{1+z}\right)  {\rm A}({\bf k}';z),
$$
where the sum on the right runs over indices ${\bf k}'$ and non-negative integers $d$ that satisfy ${\rm wt}({\bf k}')+d \leq {\rm wt}({\bf k})$, 
and $c_{{\bf k}}({\bf k}';d) $ is $\mathbb{Q}$-linear combination of multiple $T$-values of weight ${\rm wt}({\bf k})-{\rm wt}({\bf k}')-d$. 
We understand ${\rm A}(\varnothing; z)=1$, ${\rm wt}(\varnothing)=0$ for the empty index $\varnothing$, 
and the constant $1$ is regarded as a multiple $T$-value of weight $0$. 
\end{hodai}

\begin{teiri}[{\cite[Theorem 2.3]{maneka_lv2}}]\label{mLem2.3}
Let ${\bf k}$ be any index. 
The function $\psi({\bf k};s)$ can be written in terms of multiple zeta functions of level $2$ as
$$
\psi({\bf k};s)
=\sum_{{\bf k}',d\geq0} c_{{\bf k}}({\bf k}';d) \binom{s+d-1}{d}T({\bf k}';s+d).
$$
Here, the sum on the right runs over indices ${\bf k}'$ and non-negative integers $d$ that satisfy ${\rm wt}({\bf k}')+d \leq {\rm wt}({\bf k})$, 
and $c_{{\bf k}}({\bf k}';d) $ is $\mathbb{Q}$-linear combination of multiple $T$-values of weight ${\rm wt}({\bf k})-{\rm wt}({\bf k}')-d$. 
\end{teiri}

\subsection{Level $2$ versions of main results}


Our main results in this section are the following. 
These give level 2 versions of Theorem \ref{mainThm1}, \ref{mainThm2} and Corollary \ref{mainCor}. 

\begin{teiri}\label{mainThm1lv2}
Let ${\bf k}=(\{1\}^{a_1-1},b_1+1,\ldots,\{1\}^{a_n-1},b_n+1)$ with $b_n \geq 0$ be any index. 
The function $\xi({\bf k};s)$ can be written in terms of multiple zeta functions as
\begin{align*}
&\psi({\bf k};s) \\
&=(1-\delta_{0,b_n})T({\bf k}) T(s) 
   -\sum_{l=1}^{n}  \sum_{j=0}^{b_l-2} (-1)^{j+{\rm wt}({\bf k}^l)} 
T({\bf k}_{l-1},\{1\}^{a_l-1},b_l-j)T((j+1,{\bf k}^l)^\vee;s) \\
&\quad +\sum_{l=1}^{n} (-1)^{b_l+{\rm wt}({\bf k}^l)} 
 \sum_{d= 0}^{a_l} 
 \sum_{\substack{{\rm wt}({\bf e}_1)+{\rm wt}({\bf e}_2)+d=a_l \\ {\rm dep}({\bf e}_1)=n_1,\ {\rm dep}({\bf e}_2)=n_2}}
  (-1)^{{\rm wt}({\bf e}_1)} b(({\bf k}_{l-1})^\dagger;{\bf e}_1) \\
&\quad\quad\times   
T(({\bf k}_{l-1})^\dagger+{\bf e}_1) b((b_l,{\bf k}^l)^\vee;{\bf e}_2)\binom{s+d-1}{d}T((b_l,{\bf k}^l)^\vee+{\bf e}_2;s+d),   
\end{align*}
where the sum on the right runs over all non-negative indices ${\bf e}_1$, ${\bf e}_2$ that satisfy ${\rm wt}({\bf e}_1)+{\rm wt}({\bf e}_2)+d=a_l$, $n_1={\rm dep}(({\bf k}_{l-1})^\dagger)$ and $n_2={\rm dep}((b_l,{\bf k}^l)^\vee)$.  
\end{teiri}


\begin{teiri}\label{mainThm2lv2}
Let ${\bf k}=(\{1\}^{a_1-1},b_1+1,\ldots,\{1\}^{a_n-1},b_n+1)$ with $b_n \geq 0$ be any index. 
Then we have
\begin{align*}
& {\rm A}\left({\bf k};\frac{1-z}{1+z}\right) \\
%
&=(1-\delta_{0,b_n})T({\bf k})
 -\sum_{l=1}^{n}  \sum_{j=0}^{b_l-2} (-1)^{j+{\rm wt}({\bf k}^l)} 
T({\bf k}_{l-1},\{1\}^{a_l-1},b_l-j){\rm A}((j+1,{\bf k}^l)^\vee;z) \\
&\quad +\sum_{l=1}^{n} (-1)^{b_l+{\rm wt}({\bf k}^l)} 
 \sum_{d= 0}^{a_l} 
 \sum_{\substack{{\rm wt}({\bf e}_1)+{\rm wt}({\bf e}_2)+d=a_l \\ {\rm dep}({\bf e}_1)=n_1,\ {\rm dep}({\bf e}_2)=n_2}}
  (-1)^{{\rm wt}({\bf e}_1)} b(({\bf k}_{l-1})^\dagger;{\bf e}_1) \\
&\quad\quad\times   
T(({\bf k}_{l-1})^\dagger+{\bf e}_1) b((b_l,{\bf k}^l)^\vee;{\bf e}_2){\rm A}\left(\{1\}^{d};\frac{1-z}{1+z}\right)
  {\rm A}((b_l,{\bf k}^l)^\vee+{\bf e}_2;z). 
\end{align*}
where the sum on the right runs over all non-negative indices ${\bf e}_1$, ${\bf e}_2$ that satisfy ${\rm wt}({\bf e}_1)+{\rm wt}({\bf e}_2)+d=a_l$, $n_1={\rm dep}(({\bf k}_{l-1})^\dagger)$ and $n_2={\rm dep}((b_l,{\bf k}^l)^\vee)$.  
\end{teiri}


\begin{kei}\label{mainCorlv2}
For any index ${\bf k}=(\{1\}^{a_1-1},b_1+1,\ldots,\{1\}^{a_n-1},b_n+1)$ and positive integer $m$, we have
\begin{align*}
&\psi({\bf k};m+1) - (-1)^{{\rm wt}({\bf k})-a_1} \psi (\{1\}^{m-1},\overleftarrow{(b_1,{\bf k}^1)_+};a_1+1) \\
&=(1-\delta_{0,b_n})T({\bf k})  T(m+1) \\
&\quad -\sum_{l=1}^{n} \sum_{j=0}^{b_l-2} (-1)^{j+{\rm wt}({\bf k}^l)}
T({\bf k}_{l-1},\{1\}^{a_l-1},b_l-j)  T(\{1\}^{m-1},\overleftarrow{(j+2, {\bf k}^l)_+}) \\
&\quad +\sum_{l=2}^{n} \sum_{d=0}^{a_l}  (-1)^{b_l+{\rm wt}({\bf k}^l)+d} 
\psi(({\bf k}_{l-1})_-;d+1) \psi(\{1\}^{m-1},\overleftarrow{(b_l,{\bf k}^l)_+};a_l-d+1). 
\end{align*}
\end{kei}

Using the following lemmas, we can prove these results by exactly the same argument in $\S3$ and $\S4$, and hence we omit the proofs. 

\begin{hodai}[cf. {\cite[Lemma 5.1 (i)]{kt_fcn_lv2}}]\label{hodai5.1}
For any index $(k_1,\ldots,k_r)$, we have
$$
\frac{d}{dz} {\rm A}(k_1,\ldots,k_r;z)
=
\begin{cases}
\displaystyle{\frac{1}{z}{\rm A}(k_1,\ldots,k_{r-1},k_r-1;z)} &\quad (k_r\geq2), \\
\\
\displaystyle{\frac{2}{1-z^2}{\rm A}(k_1,\ldots,k_{r-1};z)} &\quad (k_r=1). 
\end{cases}
$$
\end{hodai}

\begin{hodai}[{\cite[Theorem 4.6]{CeXuZhao}}]\label{kt_fcn_lv2}
For any positive index ${\bf k}$ and any positive integer $m$, we have
\begin{equation*}\label{xirel}
\psi({\bf k};m) = \sum_{\substack{{\rm wt}({\bf j})=m-1 \\ {\rm dep}({\bf j})=n}} b(({\bf k}_+)^\dagger;{\bf j}) T(({\bf k}_+)^\dagger+{\bf j})
\end{equation*}
where sum is over all non-negative indices ${\bf j}$ of weight $m-1$ and depth $n={\rm dep}(({\bf k}_+)^\dagger)$. 
\end{hodai}

Replacing $\omega_1(t)= dt/(1-t)$ with $dt/(1-t^2)$, we also have integral representations of $\psi\left({\bf k};m\right)$ associated with $2$-posets ({\cite[p.3128]{CeXuZhao}}) and the following lemma. 

\begin{hodai}
For  any admissible index ${\bf l}$ and non-negative integer $a$, we have 
\begin{align*}
&\lim_{z\to 1 }\left\{ {\rm A}({\bf l},\{1\}^a;z) - \sum_{d= 1}^a (-1)^{a-d} \sum_{\substack{{\rm wt}({\bf e})+d=a \\ {\rm dep}({\bf e})=n }}  b({\bf l}^\dagger;{\bf e}) T({\bf l}^\dagger+{\bf e}) {\rm A}(\{1\}^d;z) \right\} \\
&= (-1)^a \sum_{\substack{{\rm wt}({\bf e})=a \\ {\rm dep}({\bf e})=n }}  b({\bf l}^\dagger;{\bf e}) T({\bf l}^\dagger+{\bf e}).
\end{align*}
\end{hodai}

\noindent {\bf Acknowledgements.}
The author would like to thank Hirofumi Tsumura, Masanobu Kaneko, Yasuo Ohno, Kentaro Ihara, Maki Nakasuji and Yuna Baba for their various comments. 
The author also thanks Yu Katagiri for his careful reading this manuscript.
 This work was supported in part by JSPS KAKENHI Grant Numbers JP19K23396. 




\begin{bibdiv}
\begin{biblist}
\bib{ak_fcn}{article}{
   author={Arakawa, Tsuneo},
   author={Kaneko, Masanobu},
   title={Multiple zeta values, poly-Bernoulli numbers, and related zeta
   functions},
   journal={Nagoya Math. J.},
   volume={153},
   date={1999},
   pages={189--209},
}
\bib{hof_1992}{article}{
   author={Hoffman, Michael E.},
   title={Multiple harmonic series},
   journal={Pacific J. Math.},
   volume={152},
   date={1992},
   number={1},
   pages={275--290},
}
\bib{hof_1997}{article}{
   author={Hoffman, Michael E.},
   title={The algebra of multiple harmonic series},
   journal={J. Algebra},
   volume={194},
   date={1997},
   number={2},
   pages={477--495},
}
\bib{ikz}{article}{
   author={Ihara, Kentaro},
   author={Kaneko, Masanobu},
   author={Zagier, Don},
   title={Derivation and double shuffle relations for multiple zeta values},
   journal={Compos. Math.},
   volume={142},
   date={2006},
   number={2},
   pages={307--338},
}
\bib{kaneko_pbn}{article}{
   author={Kaneko, Masanobu},
   title={Poly-Bernoulli numbers},
   journal={J. Th\'{e}or. Nombres Bordeaux},
   volume={9},
   date={1997},
   number={1},
   pages={221--228},
}
\bib{kt_fcn}{article}{
   author={Kaneko, Masanobu},
   author={Tsumura, Hirofumi},
   title={Multi-poly-Bernoulli numbers and related zeta functions},
   journal={Nagoya Math. J.},
   volume={232},
   date={2018},
   pages={19--54},
}
\bib{kt_fcn_lv2}{article}{
   author={Kaneko, Masanobu},
   author={Tsumura, Hirofumi},
   title={Zeta functions connecting multiple zeta values and poly-Bernoulli
   numbers},
   conference={
      title={Various aspects of multiple zeta functions---in honor of
      Professor Kohji Matsumoto's 60th birthday},
   },
   book={
      series={Adv. Stud. Pure Math.},
      volume={84},
      publisher={Math. Soc. Japan, Tokyo},
   },
   date={[2020] \copyright 2020},
   pages={181--204},
}
\bib{kt_fcn_mzv_lv2}{article}{
   author={Kaneko, Masanobu},
   author={Tsumura, Hirofumi},
   title={On multiple zeta values of level two},
   journal={Tsukuba J. Math.},
   volume={44},
   date={2020},
   number={2},
   pages={213--234},
}
\bib{ky}{article}{
   author={Kaneko, Masanobu},
   author={Yamamoto, Shuji},
   title={A new integral-series identity of multiple zeta values and
   regularizations},
   journal={Selecta Math. (N.S.)},
   volume={24},
   date={2018},
   number={3},
   pages={2499--2521},
}
\bib{ko_poset}{article}{
   author={Kawasaki, Naho},
   author={Ohno, Yasuo},
   title={Combinatorial proofs of identities for special values of
   Arakawa-Kaneko multiple zeta functions},
   journal={Kyushu J. Math.},
   volume={72},
   date={2018},
   number={1},
   pages={215--222},
}
\bib{maneka_lv2}{article}{
   author={Pallewatta, Maneka},
   title={Level two generalization of Arakawa-Kaneko zeta function and poly-cosecant numbers},
   conference={
      title={Various aspects of multiple zeta values},
   },
   book={
      series={RIMS K\^{o}ky\^{u}roku},
   },
   date={2019},
   pages={104--113},
}
\bib{y1}{article}{
   author={Yamamoto, Shuji},
   title={Multiple zeta-star values and multiple integrals},
   conference={
      title={Various aspects of multiple zeta values},
   },
   book={
      series={RIMS K\^{o}ky\^{u}roku Bessatsu, B68},
      publisher={Res. Inst. Math. Sci. (RIMS), Kyoto},
   },
   date={2017},
   pages={3--14},
}
\bib{cexu}{article}{
   author={Xu, Ce},
   title={Duality formulas for Arakawa-Kaneko zeta values and related variants},
   journal={Bull. Malays. Math. Sci. Soc.},
   volume={44},
   date={2021},
   number={5},
   pages={3001--3018},
}
\bib{CeXuZhao}{article}{
   author={Xu, Ce},
   author={Zhao, Jianqiang},
   title={Variants of multiple zeta values with even and odd summation
   indices},
   journal={Math. Z.},
   volume={300},
   date={2022},
   number={3},
   pages={3109--3142},
}
\bib{z_1994}{article}{
   author={Zagier, Don},
   title={Values of zeta functions and their applications},
   conference={
      title={First European Congress of Mathematics, Vol. II},
      address={Paris},
      date={1992},
   },
   book={
      series={Progr. Math.},
      volume={120},
      publisher={Birkh\"{a}user, Basel},
   },
   date={1994},
   pages={497--512},
}
\end{biblist}
\end{bibdiv}
\end{document}